\documentclass[a4paper,10pt]{amsart}

\usepackage{amsmath, amsthm, amssymb, verbatim}
\usepackage{graphicx}
\usepackage{url}
\usepackage[all]{xy}
	
\makeatletter
\newtheorem*{rep@theorem}{\rep@title}
\newcommand{\newreptheorem}[2]{%
\newenvironment{rep#1}[1]{%
 \def\rep@title{#2 \ref{##1}}%
 \begin{rep@theorem}}%
 {\end{rep@theorem}}}
\makeatother
	
\DeclareMathOperator{\Aut}{Aut}
\DeclareMathOperator{\PSL}{PSL}
\DeclareMathOperator{\Cay}{Cay}
\DeclareMathOperator{\Isom}{Isom}

\DeclareMathOperator{\BS}{BS}

\newcommand{\Af}{\mathbf{A}}
\newcommand{\Ag}{A}
\newcommand{\Bf}{\mathbf{B}}
\newcommand{\Bg}{B}

\newcommand{\Cg}{C}

\newcommand{\Dg}{D}
\newcommand{\NN}{\mathbf{N}}

\newcommand{\ZZ}{\mathbf{Z}}
\newcommand{\RR}{\mathbf{R}}
\newcommand{\CC}{\mathbf{C}}

\newcommand{\QQ}{\mathbf{Q}}
\newcommand{\Trivial}{\mathbf{1}}

\newcommand{\Tree}{\mathcal{T}}

\newcommand{\injnearrow}{\mathrel{\scalebox{1.2}{\rotatebox[origin=c]{45}{$\hookrightarrow$}}}}
\newcommand{\injnwarrow}{\mathrel{\scalebox{1.2}{\rotatebox[origin=c]{-45}{$\hookleftarrow$}}}}
\newcommand{\surjnearrow}{\mathrel{\scalebox{1.2}{\rotatebox[origin=c]{45}{$\twoheadrightarrow$}}}}
\newcommand{\surjnwarrow}{\mathrel{\scalebox{1.2}{\rotatebox[origin=c]{-45}{$\twoheadleftarrow$}}}}

\newtheorem{theorem}{Theorem}[section]
\newreptheorem{theorem}{Theorem}
\newreptheorem{prop}{Proposition}
\newreptheorem{corollary}{Corollary}
\newtheorem{alphtheorem}{Theorem}

\newtheorem{lemma}[theorem]{Lemma}
\newtheorem{prop}[theorem]{Proposition}
\newtheorem{corollary}[theorem]{Corollary}
\newtheorem{alphcorollary}[alphtheorem]{Corollary}

\theoremstyle{remark}
\newtheorem{step}{Step}
\newtheorem{remark}[theorem]{Remark}
\newtheorem{examples}[theorem]{Examples}
\newtheorem{notation}[theorem]{Notation}
\newtheorem{definition}[theorem]{Definition}

\title{Commability of groups quasi-isometric to trees}
\author{Mathieu Carette}
\address{\tt Universit\'e catholique de Louvain, 
IRMP,
Chemin du Cyclotron 2, bte L7.01.01, 
1348 Louvain-la-Neuve,
Belgium}
\email{\tt mathieu.carette@uclouvain.be}
\thanks{The author is a Postdoctoral Researcher of the F.R.S.-FNRS (Belgium).}

\begin{document}
	
	\begin{abstract}
		Commability is the finest equivalence relation between locally compact groups such that $G$ and $H$ are equivalent whenever there is a continuous proper homomorphism $G \to H$ with cocompact image. Answering a question of Cornulier, we show that all non-elementary locally compact groups acting geometrically on locally finite simplicial trees are commable, thereby strengthening previous forms of quasi-isometric rigidity for trees. We further show that 6 homomorphisms always suffice, and provide the first example of a pair of locally compact groups which are commable but without commation consisting of less than 6 homomorphisms. Our strong quasi-isometric rigidity also applies to products of symmetric spaces and Euclidean buildings, possibly with some factors being trees.
	\end{abstract}
		
	\maketitle
	
	\section{Results}
		
	Commability between locally compact (l.c.) groups was introduced by Cornulier in his systematic study of quasi-isometry classes of focal hyperbolic groups. Two l.c.\ groups $G, H$ are \textbf{commable} if there is a sequence of l.c.\ groups (called a \textbf{commation}) $G = G_0 - G_1 - G_2 - \ldots - G_n = H$ where a dash $G_i - G_{i+1}$ indicates there is a continuous, proper homomorphism with (closed) cocompact image (or \textbf{copci} homomorphism for short) between them, pointing in either direction. The \textbf{commability distance} (or simply distance) between two commable l.c.\ groups $G$ and $H$ is the smallest $n$ such that there is a commation as above. In order to improve readability, we write e.g. $G \nearrow \nwarrow \nearrow H$ to mean a commation $G \rightarrow G_1 \leftarrow G_2 \rightarrow H$.
	
	Commability between compactly generated l.c.\ groups implies quasi-isometry. A natural questions promptly arises: does the converse hold for specific classes of groups? Positive examples are provided by the work of Kleiner-Leeb (see Theorem~\ref{theorem:Kleiner-Leeb} below when no factor is a tree). On the other hand the author and Tessera \cite{CaretteTessera13} provide examples of groups which are quasi-isometric but not commable, such as the free products $\Gamma_1 \ast \ZZ$ and $\Gamma_2 \ast \ZZ$ where the $\Gamma_i$ are non-commensurable cocompact lattices in $\PSL_2(\CC)$. Cornulier conjectured {\cite[Conjecture 6.A.1]{Cornulier12}} that if $G$ is a focal l.c.\ group and $H$ is a compactly generated l.c.\ group quasi-isometric to $G$ then $G$ is commable to $H$. In the case where $G$ is a totally disconnected focal group, it acts geometrically on a regular tree.
	
	The purpose of the present paper is a detailed study of the commability relation among groups which are quasi-isometric to trees. More precisely, denote by $\Tree$ the class of all non-elementary locally compact groups acting geometrically on a locally finite tree (e.g. a non-abelian finitely generated free group, $\PSL_2(\ZZ)$, $\PSL_2(\QQ_p)$, the automorphism group of a $d$-regular tree for $d\geq 3$, \ldots). Groups in $\Tree$ form a single quasi-isometry class and in particular $\Tree$ is closed under commability. 
	
	Cornulier asked in \cite[Question 5.E.7]{Cornulier12} and \cite[Question 7.7]{Cornulier13} whether all groups in $\Tree$ are commable to each other. We provide a positive answer, thereby solving Cornulier's Conjecture for focal groups of totally disconnected type.
	\begin{alphtheorem} \label{theorem:upper_bounds} There exists a group $G_0 \in \Tree$ such that for any $G, G' \in \Tree$ and any unimodular group $U \in \Tree$, there are commations 
	\begin{enumerate}
		\item $G \nearrow\nwarrow\nearrow\nwarrow G_0$; \label{item:upper_bound_radius}
		\item $G \nearrow\nwarrow\nearrow\nwarrow\nearrow U$; \label{item:upper_bound_unimod}
		\item $G \nearrow\nwarrow\nearrow\nwarrow\nearrow\nwarrow G'$. \label{item:upper_bound_diameter}
	\end{enumerate}
	\end{alphtheorem}	
	\begin{corollary}\label{corollary:tree_commable} Any two groups in $\Tree$ are commable.
	\end{corollary}

	Combining Corollary~\ref{corollary:tree_commable} with the fact that $\Tree$ is closed under quasi-isometry, we obtain a stronger, algebraic form of quasi-isometric rigidity. 
	\begin{corollary}\label{corollary:tree_strong_QI} If $G, H$ are compactly generated l.c.\ groups quasi-isometric to a locally finite bushy tree, then $G$ and $H$ are commable.
	\end{corollary}

	Note that the only discrete groups in $\Tree$ are finitely generated virtually \{free non-abelian\} and hence are all commensurable. More generally Corollary~\ref{corollary:tree_commable} is well-known for unimodular groups. Indeed we have the following result of Bass and Kulkarni
	\begin{theorem}[\cite{BassKulkarni90}]\label{theorem:BassKulkarni} Any unimodular group $U \in \Tree$ contains a cocompact lattice $\Gamma < U$.
	\end{theorem}
	Such $\Gamma$ is virtually free, and it follows easily that given any two unimodular groups $U,U' \in \Tree$ there is a commation $U \nwarrow \nearrow U'$.

	However, commations in $\Tree$ may be significantly more complicated than when restricted to unimodular groups. Indeed, we provide examples showing that the distances in Theorem~\ref{theorem:upper_bounds} are optimal. In particular, we give the first examples of l.c.\ groups which are commable but at commability distance greater than $4$.
	\begin{alphtheorem} \label{theorem:lower_bounds}
		Let $G \in \Tree$ and let $U \in \Tree$ be unimodular. There are groups $H_1, H_2 \in \Tree$ such that:
		\begin{enumerate}
			\item $H_1$ is at commability distance at least $4$ from $G$. \label{item:lower_bound_radius}
			\item $H_1$ is at commability distance at least $5$ from $U$.\label{item:lower_bound_unimod}
			\item $H_1$ is at commability distance at least $6$ from $H_2$.\label{item:lower_bound_diameter}
		\end{enumerate}
	\end{alphtheorem}
	\begin{remark} Commability distance forgets the information contained in the directions of arrows. Indeed, the relation of commability through $\nearrow \nwarrow$ and $\nwarrow \nearrow$ are very different. For example Mosher, Sageev and Whyte gave examples of virtually free groups which are not commable through $\nearrow \nwarrow$ despite being commensurable \cite{MosherSageevWhyte03}. In fact if $G, U, H_1, H_2$ are as in Theorem~\ref{theorem:lower_bounds} we show that any commation of minimal length from $H_1$ to $G$ (resp. $U$ and $H_2$) is of the form $H_1\nearrow\nwarrow\nearrow\ldots G$ (resp. $H_1\nearrow\nwarrow\nearrow\ldots U$ and $H_1\nearrow\nwarrow\nearrow\ldots H_2$).
	\end{remark}
	Combining Theorems~\ref{theorem:upper_bounds} and~\ref{theorem:lower_bounds} we thus obtain:
	\begin{corollary} The commability class $\Tree$ has radius $4$ and diameter $6$. For every unimodular group $U \in \Tree$ the smallest ball around $U$ covering $\Tree$ has radius $5$.
	\end{corollary}
	\begin{remark} The question of existence of a pair of groups at commability distance greater than $4$ was raised in \cite[Remark 5.A.2]{Cornulier12}. If two discrete groups are commable within discrete groups then they are commensurable up to finite kernels (i.e. there is a commation within discrete groups of the form $\injnwarrow \surjnearrow \surjnwarrow \injnearrow$).
	\end{remark}
	 
	\subsection{Products of symmetric spaces and Euclidean buildings} \label{sec:products} Throughout this section, we fix $X = \prod_{i=1}^{m} X_i \times Y$ where 
	\begin{itemize}
		\item each $X_i$ is a locally finite bushy tree
		\item $Y = \prod_{j=1}^{n} Y_j$ and each $Y_j$ is one of the following
	\begin{itemize}
		\item an irreducible symmetric space of noncompact type.
		\item a thick irreducible Euclidean building of rank/dimension $\geq 2$, with cocompact Weyl group.
	\end{itemize}
	\end{itemize}
	Suppose additionally that the factors are scaled in such a fashion that whenever $Y_j$ is homothetic to $Y_k$ then $Y_j$ is isometric to $Y_k$. The work of Cornulier \cite{Cornulier12} allows us to recast a result of Kleiner-Leeb in the setting of locally compact groups as follows:
	\begin{theorem}[\cite{KleinerLeeb09}] \label{theorem:Kleiner-Leeb} Let $X$ be as above, and let $G$ be a compactly generated l.c.\ group quasi-isometric to $X$. Then there is a family $X_i'$ of locally finite bushy trees such that $G$ acts geometrically by isometries on $X'=\prod_{i=1}^m X_i' \times Y$.
	\end{theorem}
	
	In particular if $m=0$ then $X'=Y = X$ so that there is a copci homomorphism $G \to \Isom(X)$. On the other hand, if $m \geq 1$ then the situation is drastically different. As a basic example, the free group $F_2$ of rank $2$ does not act geometrically on the $5$-regular tree $T_5$ despite being quasi-isometric to it. Using Theorem~\ref{theorem:Kleiner-Leeb} we nonetheless extend Corollary~\ref{corollary:tree_strong_QI} to products as above.
	
	\begin{alphcorollary}\label{corollary:products} Let $X$ be as above. If $G, H$ are compactly generated l.c.\ groups quasi-isometric to $X$ then $G$ and $H$ are commable. 
	\end{alphcorollary}
	This applies in particular to (cocompact S-arithmetic lattices in) products of simple algebraic groups over local fields, possibly with non-Archimedean rank $1$ factors (e.g. $\PSL_2(\QQ_p)$). Note also that there exist many discrete groups quasi-isometric to the product of two bushy trees which are not commensurable with products of free groups (e.g. irreducible cocompact lattices in $\PSL_2(\QQ_p) \times \PSL_2(\QQ_p)$). 
	
	The paper is organized as follows. We provide more detailed statements which imply Theorems~\ref{theorem:upper_bounds} and~\ref{theorem:lower_bounds} in Section~\ref{sec:outline}. We review the basics of commability in Sections~\ref{sec:lc_groups} and~\ref{sec:commability}. Roughly speaking, studying commability in the class $\Tree$ requires a combined understanding of the two following situations:
	\begin{itemize}
		\item Which groups $G_1, G_2$ act geometrically on the same tree  $G_1 \curvearrowright T \curvearrowleft G_2$?
		\item Which trees $T_1, T_2$ admit geometric actions by a same group $T_1 \curvearrowleft G \curvearrowright T_2$?
	\end{itemize}
	In the first situation, $G_1$ and $G_2$ are both cocompact subgroups of $\Isom(T)$. The study of cocompact subgroups of $\Isom(T)$ is addressed via covering theory of edge-indexed graphs, which is the content of Sections~\ref{sec:gog}--\ref{sec:covering}. Dealing with the second situation is done using Forester's Deformation Theorem and developed in Section~\ref{sec:deformations}.
	 Finally, we combine these tools in the technical Sections~\ref{sec:upper_bounds},~\ref{sec:key_examples} and~\ref{sec:lower_bounds}, an overview of which is deferred to Section~\ref{sec:outline}.
	\subsection*{Acknowledgments}
	I am grateful to Roberto Tauraso for providing a proof of Proposition~\ref{prop:finite_series}. I also thank Yves de Cornulier and Romain Tessera for inspiring discussions. Finally, I thank the referee for his careful reading and useful comments.
		
	\section{Detailed outline} \label{sec:outline}
	
	We record here the key statement to be proven, and apply them to prove Theorems~\ref{theorem:upper_bounds} and~\ref{theorem:lower_bounds}. We end with the proof of Corollary~\ref{corollary:products}.
	
	\subsection{Commability and upper bounds}
	
	For any natural number $n \geq 3$ let $T_n$ denote the $n$-regular tree. The first step in proving upper bounds is to show
	
	\begin{reptheorem}{theorem:commation_to_regular}
	Given any $G \in \Tree$ there is some $n \geq 3$ such that there is a commation $G \nearrow \nwarrow \nearrow \Aut(T_n)$.
	\end{reptheorem}
	Recall that by Theorem~\ref{theorem:BassKulkarni} of Bass--Kulkarni, any two unimodular groups $U, U' \in \Tree$ are commable through $U \nwarrow \nearrow U'$. Moreover the full automorphism group of the $n$-regular tree is unimodular. Thus Theorem~\ref{theorem:upper_bounds}\eqref{item:upper_bound_unimod} follows from Theorem~\ref{theorem:commation_to_regular}. 
		
	It is not possible to improve the above result by requiring $n$ to be an a priori fixed integer. However, it is possible to change slightly the commation above (and the number $n$) in order to obtain the upper bound of $6$ on the diameter of $\Tree$:
	\begin{reptheorem}{theorem:upper_bound_diam}[Theorem~\ref{theorem:upper_bounds}\eqref{item:upper_bound_diameter}] Given any $G, G' \in \Tree$ there is some $m \geq 3$ such that there is a commation $G \nearrow \nwarrow \nearrow \Aut(T_{m}) \nwarrow \nearrow \nwarrow G'$.
	\end{reptheorem}
	
	In order to show that the radius of $\Tree$ is at most $4$, one must choose an appropriate candidate group from which to measure distances. It follows from Theorem~\ref{theorem:lower_bounds}\eqref{item:lower_bound_unimod} that such a candidate cannot be chosen to be unimodular.
		
	Let $G_{2,4}$ be the fundamental group of the following graph of groups
	\[\begin{xy}\xymatrix{ 
			 \ZZ_2 \ar@{-}@(ur,dr)^<<{\times 2}^>>{\times 4}^{\ZZ_2}
			 }
			 \end{xy}\]
	where $\ZZ_2$ denotes the (compact) group of $2$-adic integers. The group $G_{2,4}$ is isomorphic to the closure of the Baumslag-Solitar group $\BS(2,4)$ in the automorphism group of the standard Bass-Serre tree of $\BS(2,4)$, which is the $6$-regular tree.

	The upper bound of $4$ on the radius of $\Tree$ is provided by the following
	\begin{reptheorem}{theorem:upper_bound_radius}[Theorem~\ref{theorem:upper_bounds}\eqref{item:upper_bound_radius}]
		Given any $G \in \Tree$ there is a commation $G \nearrow \nwarrow \nearrow \nwarrow G_{2,4}$.
	\end{reptheorem}
	
	\subsection{Examples attaining the upper bounds}
	
	We show that all upper bounds above are attained, using the following family of groups.
	Fix $q,r,s$ pairwise distinct odd primes. Let $n = 2 \times q \times r \times s$ and $\ZZ_n = \ZZ_2 \times \ZZ_q \times \ZZ_r \times \ZZ_s$. We let $H^{2,2,2}_{q,r,s}$ denote the fundamental group of the following graph of groups:
	 
	\[\begin{xy}
			\xymatrix{ \ZZ_n \ar@{-}[rr]^<<{\times 2}^>>{\times q}^{\ZZ_n} & & \ZZ_n \ar@{-}[dl]^<<{\times 2}^>>{\times s}^{\ZZ_n} \\ 
			  & \ZZ_n \ar@{-}[lu]^<<{\times 2}^>>{\times r}^{\ZZ_n} &
			 }\end{xy}.\]
	Let $T^{2,2,2}_{q,r,s}$ denote the Bass-Serre tree of the above graph of groups. The group $H^{2,2,2}_{q,r,s}$ is isomorphic to $\overline{\Gamma^{2,2,2}_{q,r,s}}$ the closure of the obvious generalized Baumslag-Solitar group (fundamental group of the obvious graph of infinite cyclic groups) in $\Aut(T^{2,2,2}_{q,r,s})$.
	
	In Definition~\ref{definition:S_2,N} we associate to each natural number $N\geq 1$ a set $\mathcal S_{2,N} \subset \NN$ which turns out to be sparse even when intersected with the set of primes.
	\begin{repcorollary}{corollary:primes_not_in_S2N} Given any $N \in \NN$ there are infinitely many primes not in $\mathcal S_{2,N}$.
	\end{repcorollary}
	Our main technical result shows that the groups $H^{2,2,2}_{q,r,s}$ satisfy very strong rigidity properties.
	\begin{reptheorem}{theorem:key_example} Let ${q,r,s} \geq 11$ be pairwise distinct primes. 
		\begin{enumerate} 
			\item Suppose there is a commation $H^{2,2,2}_{q,r,s} \nwarrow \nearrow \nwarrow G$ then there is a commation $H^{2,2,2}_{q,r,s} \nearrow \Aut(T^{2,2,2}_{q,r,s}) \nwarrow G$.
			\item Suppose moreover that $q,r,s$ are not in $\mathcal S_{2,N}$. Then $G$ cannot act geometrically on a tree with degree bounded above by $N$.
		\end{enumerate}
	\end{reptheorem}

	In order to show that the diameter of $\Tree$ is $6$ it is enough to show the following:
	\begin{repcorollary}{corollary:lower_bound_diameter} Let $q,r,s$ be sufficiently large primes, let $N = \max\{q,r,s\}$ and let $q',r',s'$ be primes not in $\mathcal S_{2,N}$. Then there is no commation of length less than $6$ from $H^{2,2,2}_{q,r,s}$ to $H^{2,2,2}_{q',r',s'}$.
	\end{repcorollary}
	
	We also show that the radius of $\Tree$ is exactly $4$ as follows.	
	
	\begin{repcorollary}{corollary:lower_bound_radius}Let $G \in \Tree$ and let $T$ be a geometric $G$-tree. Let $N$ be the maximal degree of a vertex of $T$, and let $q,r,s$ be primes not in $\mathcal S_{2,N}$. Then $G$ is at distance at least $4$ from $H^{2,2,2}_{q,r,s}$.
	\end{repcorollary}
	
	\begin{repprop}{prop:bound_on_degree}  Let $U\in \Tree$ be a unimodular group. Then there is a natural number $N_U$ such that if $T$ is a geometric minimal $U$-tree then all vertices of $T$ have degree at most $N_U$.
	\end{repprop}
	\begin{repcorollary}{corollary:lower_bound_unimod} Let $U \in \Tree$ be a unimodular group and $N_U \in \NN$ be as in Proposition~\ref{prop:bound_on_degree}. For pairwise distinct primes $q,r,s \geq 11$ not in $\mathcal S_{2,N_U}$ the group $H^{2,2,2}_{q,r,s}$ is at distance at least $5$ from $U$.
	\end{repcorollary}
	
	Theorem~\ref{theorem:lower_bounds} now follows from the combination of Corollaries~\ref{corollary:lower_bound_diameter},~\ref{corollary:lower_bound_radius} and~\ref{corollary:lower_bound_unimod} (along with the fact that $\mathcal S_{2,N} \subset \mathcal S_{2,N'}$ for $N \leq N'$). 
	
	\begin{remark} It is straightforward to strengthen Theorem~\ref{theorem:lower_bounds} as follows. 
	
	There is a countable family of groups $\{H_k\}_{k \in \NN}$ in $\Tree$ such that: 
		\begin{enumerate}
			\item Every unimodular group in $\Tree$ is at distance at least 3 from $H_k$ for all $k$.
			\item For every group $G \in \Tree$ there is some $K_G \in \NN$ such that $G$ is at distance at least $4$ from $H_k$ for all $k\geq K_G$.
			\item For every unimodular group $U \in \Tree$ there is some $K_U \in \NN$ such that $U$ is at distance $5$ from $H_k$ for all $k\geq K_U$.
			\item $H_k$ is at distance $6$ from $H_{k'}$ for $k \neq k'$.
		\end{enumerate}
	In particular one cannot cover $\Tree$ with the $2$-neighborhood of unimodular groups, nor with finitely many balls of radius $3$, nor with finitely many balls of radius $4$ centered at unimodular groups.
	\end{remark}
	
	\subsection{Proof of Corollary~\ref{corollary:products}}
	
	Let $X = \prod_{i=1}^{m} X_i \times Y$ be as in Section~\ref{sec:products}. Let $G,H$ be compactly generated l.c.\ groups quasi-isometric to $X$. By Theorem~\ref{theorem:Kleiner-Leeb} there are locally finite bushy trees $X_1', \ldots X_m'$ and $X_1'', \ldots, X_m''$ such that $G$ and $H$ act geometrically on $X' = \prod_{i=1}^{m} X_i' \times Y$ and $X'' = \prod_{i=1}^{m} X_i'' \times Y$ respectively. Moreover the group $\prod_{i=1}^{m} \Isom(X_i') \times \Isom(Y)$ has finite index in $\Isom(X')$ (and similarly for $X''$) so that there are commations 
	\[ G \nearrow \Isom(X') \nwarrow \prod_{i=1}^{m} \Isom(X_i') \times \Isom(Y) \] and \[ H \nearrow \Isom(X'') \nwarrow \prod_{i=1}^{m} \Isom(X_i'') \times \Isom(Y). \]
	Finally for each $1\leq i\leq m$ Corollary~\ref{corollary:tree_commable} provides commations from $\Isom(X_i')$ to $\Isom(X_i'')$, so that taking the product of these commations and of isomorphisms $\Isom(Y) \cong \Isom(Y)$ we find a commation from $\prod_{i=1}^{m} \Isom(X_i') \times \Isom(Y)$ to $\prod_{i=1}^{m} \Isom(X_i'') \times \Isom(Y)$.

	\section{Preliminaries}
	
	\subsection{Compactly generated l.c.\ groups} \label{sec:lc_groups}
		
	We remind the reader of a few basic definitions and facts about compactly generated locally compact groups.
	
	Let $G$ be a locally compact group acting by isometries on a proper metric space $X$. We call the action 
	\begin{itemize}
		\item \textbf{continuous} if the map $G \times X \to X$ is continuous.
		\item \textbf{proper} if for any compact $K \subset X$ the set $\{g \in G \mid gK \cap K \neq \emptyset \}$ has compact closure.
		\item \textbf{cocompact} if there is a compact $K \subset X$ such that $X = GK$.
		\item \textbf{geometric} if it is continuous, proper and cocompact.
	\end{itemize}
	
	A basic motivation for the study of locally compact groups stems from the following fact: if $X$ is a proper geodesic metric space, then the isometry group $\Isom(X)$ endowed with the topology of uniform convergence on compact sets is a locally compact topological group. Moreover, the action of $\Isom(X)$ on $X$ is proper and continuous.
	
	The Schwartz-Milnor lemma asserts that if a locally compact group $G$ acts geometrically on a proper geodesic metric space $X$, then $G$ has a compact generating set $K$ and the Cayley graph $\Cay(G,K)$ is quasi-isometric to $X$. As is the case for finitely generated groups, two compact generating sets of a locally compact group $G$ induce bilipschitz word metrics. Thus, the quasi-isometry class of a compactly generated locally compact groups is well-defined. However, Cayley graphs of nondiscrete locally compact groups with respect to compact generating sets have two shortcomings: the Cayley graph is not locally finite, and the action of $G$ on its Cayley graph is not continuous. 	It is nonetheless true that any compactly generated locally compact group acts geometrically on some proper geodesic metric space \cite[Proposition 2.1]{CCMT12}. In particular any quasi-isometry invariant of proper geodesic metric spaces (e.g. the space of ends, \ldots) turns into a quasi-isometry invariant of compactly generated locally compact groups.

	If $X$ is a proper geodesic metric space endowed with a geometric action by $G$ and $H$ is a locally compact group, then a homomorphism $H \to G$ is copci if and only if the induced action of $H$ on $X$ is geometric. Observe that being compactly generated is a commability invariant. In view of the above facts, one can alternatively describe commability in terms of geometric actions.
	\begin{prop} \label{prop:commability_spaces} Let $G, G'$ be l.c.\ groups, and suppose $G$ is compactly generated. Then
	\begin{enumerate}
		\item \label{item:commability_spaces_seq}$G,G'$ are commable if and only if there is a sequence of isometric actions \[G = G_0 \curvearrowright X_1 \curvearrowleft G_2 \curvearrowright X_3 \ldots X_{2n -1} \curvearrowleft G_{2n} = G'\]  where each $G_i$ is locally compact, each $X_j$ is a proper geodesic metric space and each action is geometric.
		\item If $G,G'$ are commable, then they are quasi-isometric.
	\end{enumerate}
	\end{prop}

	\subsection{Quasi-isometric rigidity of trees} \label{sec:commability}
	
	A locally compact group $G$ is called \textbf{non-elementary} if it does not act continuously and properly on the real line. Equivalently, $G$ is non-elementary if it is neither compact ($0$-ended) nor compactly generated and $2$-ended.

	\begin{prop} Let $G$ be a compactly generated l.c.\ group with infinitely many ends. Then $G$ has a unique maximal compact normal subgroup (namely the kernel of the action on the space of ends) which contains the connected component of the identity.
	\end{prop}
	
	A tree $T$ is called \textbf{bushy} if $T$ is unbounded and the set of vertices which separate $T$ in at least $3$ unbounded components is cobounded. Any two locally finite bushy trees of bounded degree are quasi-isometric \cite[Lemma 5.E.9]{Cornulier12}.

	We shall freely use the following equivalent characterizations of the class $\Tree$, which essentially follow from Stallings-Abels' theorem on ends of compactly generated l.c.\ groups, as well as accessibility results of Dunwoody and Thomassen-Woess.
	\begin{theorem}[{\cite[Theorem 4.A.1]{Cornulier12}}] Let $G$ be a l.c.\ group. TFAE:
	\begin{enumerate}
		\item $G$ is non-elementary and acts geometrically on a locally finite tree. \label{item:tree_action}
		\item $G$ is the fundamental group of a finite graph of compact groups with open edge monomorphisms. \label{item:gocg}
		\item $G$ is compactly generated and quasi-isometric to a locally finite bushy tree.\label{item:quasi-isometry}
	\end{enumerate}
	\end{theorem}
	Consequently if $G \in \Tree$ and $G'$ is commable to $G$ then $G' \in \Tree$. Moreover, the spaces $X_i$ as in Proposition~\ref{prop:commability_spaces}\eqref{item:commability_spaces_seq} can be chosen to be locally finite trees.
	
	\begin{remark} Weaker characterizations of virtually free groups carry over to the class $\Tree$ \cite{CaretteDreesen12}. Indeed a l.c. group $G$ is in $\Tree$ if and only if one of the equivalent conditions is satisfied:
		\begin{enumerate}
			\item $G$ is non-elementary hyperbolic with totally disconnected boundary.
			\item $G$ has a uniform convergence action on the Cantor Set.
		\end{enumerate}
	\end{remark}

	\subsection{Graphs, graphs of groups} \label{sec:gog} We fix notation for graphs of groups, and refer the reader to \cite{Serre80} for the correspondance between actions on trees and graphs of groups. A \textbf{graph} $\Ag$ is given by the following data: a vertex set $VA$, an edge set $EA$, a fixed-point-free involution $\overline{\phantom e} : EA \to EA$ which associates to each edge its \textbf{inverse} edge, and a map $\alpha : EA \to VA$ assigning to each edge its \textbf{initial vertex}. For convenience, we also denote the \textbf{terminal vertex} of an edge $e$ by $\omega(e) = \alpha(\overline{e})$.
	
	A \textbf{graph of groups} $\Af$ is given by the following data: an underlying graph $\Ag$, a vertex group $\Af_v$ for each $v\in V\Ag$, an edge group $\Af_e$ for each $e \in E\Ag$ with $\Af_e = \Af_{\overline e}$ and injective edge homomorphisms $\alpha_e : \Af_e \to \Af_{\alpha(e)}$. For convenience, we also define $\omega_e :=\alpha_{\overline e}$. Given a connected graph of groups $\Ag$ we denote by $\pi_1(\Ag)$ its fundamental group.
	
	Let $\Af$ be a connected graph of groups. If each vertex and edge group is a topological group, and each edge monomorphisms is continuous and open, then $G = \pi_1(\Af)$ is a topological group when declaring each vertex and edge group to be open in $G$, and the action of $G$ on the Bass-Serre tree is continuous. Conversely, if a topological group $G$ acts continuously and without inversions on a tree $T$ then edge stabilizers are open in $G$, and in particular for every edge $e\in ET$ the inclusion $G_e \subset G_{\alpha(e)}$ is continuous and open.
	
	\subsection{From trees to edge-indexed graphs}
	We will use notations close to Bass \cite{Bass93} and Bass-Kulkarni \cite{BassKulkarni90}.	
	An \textbf{edge-indexed graph} is a pair $(\Ag,i)$ where $\Ag$ is a graph and $i$ is a map from $EA$ to the positive natural numbers $\NN$ (excluding $0$). The \textbf{trivial indexing} on a graph $A$ is the map $\Trivial$ which assigns $1$ to every edge. In the following, we will tacitly identify a graph $A$ with the trivially indexed graph $(A, \Trivial)$.
	
	Given a graph of groups $\Af$ such that all edge monomorphisms have finite index image (for example coming from the action of a group on a locally finite tree), one associates an edge-indexed graph with the same underlying graph as $\Af$ and the indexing of $(\Ag,i)$ is defined by $i(e) = [\Af_{\alpha(e)} : \alpha_e (\Af_e)]$. Following Bass--Kulkarni, we call $\Af$ a \textbf{grouping} of $(\Ag,i)$ if $(\Ag,i)$ is obtained from $\Af$ in this way. We observe that any edge-indexed graph always admits a grouping.
	
	\begin{definition}[(Pro)cyclic grouping] \label{defn:cyclic_grouping} Given any edge-indexed graph $(\Ag,i)$ we define its \textbf{cyclic grouping} $\ZZ(\Ag,i)$ which is the graph of groups $\Af$ with underlying graph $\Ag$, with infinite cyclic vertex and edge groups, and with edge monomorphisms $\alpha_e$ corresponding to multiplication by $i(e)$. Fundamental groups of such graphs of groups are also called \emph{generalized Baumslag-Solitar groups}. Given a torsion-free procyclic group (e.g. of the form $\ZZ_n$) we define the corresponding \textbf{procyclic grouping} analogously and denote it by $\ZZ_n(\Ag,i)$.
	\end{definition}
	
	\subsection{Covering theory for edge-indexed graphs} \label{sec:covering}
	A \textbf{covering map} of edge-indexed graphs $p : (\Ag,i) \to (\Bg,j)$ is a surjective graph map $p : A \to B$ with the additional property that for any $v \in VA$ and any edge $e \in EB$ with the property that $\alpha(e) = p(v)$ we have \[\sum_{f \in p^{-1}(e) \cap \alpha^{-1}(v)} i(f) = j(e).\] Otherwise stated, for any edge $e \in EB$ and any lift $v \in VA$ of the initial vertex of $e$, the index of the edge $e$ is the sum of the indices of all lifts of that edge with initial vertex $v$. 
	
	We collect facts from Tits \cite{Tits70}, Bass \cite{Bass93} and Bass-Kulkarni \cite{BassKulkarni90}. Every connected edge-indexed graph $(\Ag,i)$ admits an (essentially unique) covering by a tree with trivial indexing, called \textbf{universal cover} and denoted by $p : \widetilde{(\Ag,i)} \to (\Ag,i)$.
	\begin{remark}[Constructing the universal cover] Here are two approaches to constructing the universal covering tree of $(\Ag,i)$. 
	
	The first is via Bass-Serre theory: choose any grouping $\Af$ of $(\Ag,i)$ (see Definition~\ref{defn:cyclic_grouping} for the existence of a grouping), let $T$ be the corresponding Bass-Serre tree, and $p : T \to \pi_1(\Af) \backslash T = \Ag$ be the quotient map. Then it follows from Bass-Serre theory that $p : (T, \Trivial) \to (\Ag,i)$ is a cover.
	 	
	Alternatively, it is possible to construct the universal covering tree by induction on $k$, starting with a vertex $T_0 = v_0$ and constructing the ball $T_k$ of radius $k$ around $v_0$ by adding edges to the leaves of $T_{k-1}$, and mapping $p_k : T_k \to \Ag$ so that the restriction of the maps to the vertices of $T_{k-1}$ is like a cover (i.e. for every vertex $v \in VT_{k-1}$ and every edge $e \in E\Ag$ starting at $p_k(v)$ there are exactly $i(e)$ lifts of $e$ starting at $v$). Taking the union of the $T_k$ and of the maps $p_k$ produces the desired universal cover.
	\end{remark} 
	\begin{remark}[Recognizing the universal covering tree] \label{remark:recognition} Let $(\Ag,i)$ be a connected edge-indexed graph, and assume for simplicity that the graph is combinatorial, i.e. there does not exist $e \neq e'$ with $\alpha(e) = \alpha(e')$ and $\omega(e) = \omega(e')$ so that each (directed) edge $e$ is determined by its endpoints $(\alpha(e), \omega(e))$. Then, forgetting the covering map, the universal covering tree of $(\Ag,i)$ can be described as the unique tree $T$ (up to isomorphism) such that there is a partition of the vertices (given by a map $p : VT \to V\Ag$) with the property that for any edge $(v,w)$ in $\Ag$, each vertex in $p^{-1}(v)$ is adjacent to exactly $i(e)$ vertices in $p^{-1}(w)$.
	\end{remark}
	
	Given a connected edge-indexed graph $(\Ag,i)$ we denote by $\pi_1(\Ag,i)$ the group of \textbf{deck-transformations} of the universal covering $p : \widetilde{(\Ag,i)} \to (\Ag,i)$ defined as the group of those automorphisms $\alpha \in \Aut(\widetilde{(\Ag,i)})$ such that $\alpha \circ p = p$. The group $\pi_1(\Ag,i)$ endowed with the topology of pointwise convergence is a locally compact group acting continuously and properly (i.e. with compact open vertex stabilizers) on $\widetilde{(\Ag,i)}$ without inversions. Moreover the quotient graph of groups $\Af$ is naturally isomorphic as a graph to $A$, and for every edge $e \in EA$ one has $i(e) = [\Af_{\alpha(e)} : \alpha_e (\Af_e)]$.
	
	\begin{lemma} \label{lemma:qfree_group_copci_elg} Let $G$ be a l.c.\ group acting geometrically without inversions on a locally finite tree $T$. Let $(\Ag,i)$ be the corresponding (finite, connected) edge-indexed graph. Then there is copci homomorphism $\varphi : G \to \pi_1(\Ag,i)$ and a $\varphi$-equivariant isomorphism $\varphi : T \to \widetilde{(\Ag,i)}$.
	\end{lemma}
	Indeed, the tree $T$ endowed with the trivial indexing covers the edge-indexed graph $(\Ag,i)$, and identifying $T$ with the universal cover $\widetilde{(\Ag,i)}$ yields the desired homomorphism $\varphi : G \to \pi_1(\Ag,i)$.

	\begin{lemma}[{\cite[Lemma 18]{MosherSageevWhyte02}}] \label{lemma:covering} Any covering map $p : (\Ag,i) \to (\Bg,j)$ of finite connected edge-indexed graphs induces a copci homomorphism $p^* : \pi_1(\Ag,i) \to \pi_1(\Bg,j)$.
	\end{lemma}
	
	Remark that conversely, copci homomorphisms induce covering maps between edge-indexed graphs as follows. Let $G$ be a l.c.\ group acting geometrically without inversions on a locally finite tree $T$ and let $(\Ag,i)$ be the quotient edge-indexed graph. Then any copci homomorphism $H \to G$ induces a geometric action of $H$ on $T$, and the quotient edge-indexed graph $(\Bg,j)$ is finite and covers $(\Ag,i)$ (by mapping each $H$-orbit to the $G$-orbit that contains it).
	
	\subsection{The modular homomorphism}
	Let $(\Ag,i)$ be a finite connected edge-indexed graph. Given a combinatorial path $e_1 e_2 \ldots e_n$ define \[\Delta(e_1 e_2 \ldots e_n) = \frac{i(e_1) i(e_2) \ldots i(e_n)}{i(\overline{e_1}) i(\overline{e_2}) \ldots i(\overline{e_n})}.\]	Clearly $\Delta$ is invariant under homotopy of paths relative to the endpoints, so that picking a basepoint $v \in V\Ag$ yields a homomorphism $\Delta : \pi_1(\Ag,v) \to \QQ^\times$ called the \textbf{modular homomorphism}.
	
	Let $T$ be a geometric $G$-tree without inversions, let $\Af$ denote the quotient graph of groups and let $(\Ag,i)$ be the quotient edge-indexed graph. Then the modular homomorphism defined above agrees with the modular homomorphism $\Delta^G : G \to \RR^\times$ in the sense of Haar measure. Indeed, there is a quotient map $p : G = \pi_1(\Af) \to \pi_1(\Ag)$ and the modular functions satify $\Delta^G = \Delta \circ p$ as shown in \cite[(3.6)]{BassKulkarni90}.
	
	We call $G$ and $(\Ag,i)$ \textbf{unimodular} if the modular homomorphism is trivial. Besides Theorem~\ref{theorem:BassKulkarni} we shall only need the following simple observations. If $\Ag$ is a tree, then $(\Ag,i)$ is unimodular. In particular, the automorphism group of a (bi)regular tree is unimodular.  
		
	\section{Deformations of locally finite trees} \label{sec:deformations}
	
	\subsection{Collapsing an edge with trivial index} 
	Let $(\Ag,i)$ be an edge-indexed graph, and let $e$ be a non-loop edge (i.e. $\alpha(e) \neq \omega(e)$) with $i(e) = 1$. We say that $(\Bg,j)$ is obtained from $(\Ag,i)$ by \textbf{collapsing} $e$ if the following holds : 
	\begin{itemize}
		\item $V\Bg = V\Ag/(\alpha(e) \sim \omega(e))$
		\item $E\Bg = E\Ag \backslash \{e, \overline e\}$
		\item Edge inversion in $\Bg$ is the restriction of edge inversion in $\Ag$.
		\item Initial and terminal vertices in $\Bg$ are defined naturally via the corresponding maps in $\Ag$ composed with the natural quotient map $V\Ag \to V\Bg$.
		\item Let $f \in E\Ag \backslash \{e, \overline e\}$. If $\alpha(f) = \alpha(e)$ then define $j(f) = i(f) i(\overline e)$. Otherwise define $j(f) = i(f)$.
	\end{itemize}
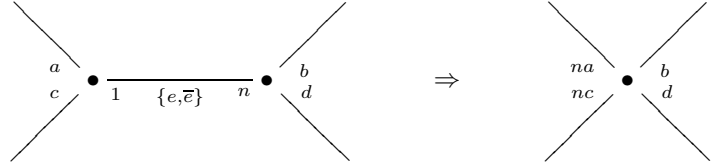
	\[ \begin{xy}
	\xymatrix{ \ar@{-}[dr]_>>>{a} & & & & \ar@{-}[dl]^>>>{b} & & \ar@{-}[dr]_>>>{na} & & \ar@{-}[dl]^>>>{b} \\ 
	& \bullet \ar@{-}[rr]_<<{1}_>>{n}_{\{e,\overline e\}} & &  \bullet &  &  \Rightarrow & & \bullet \\
	\ar@{-}[ur]^>>>{c} & & & & \ar@{-}[ul]_>>>{d} & & \ar@{-}[ur]^>>>{nc} & & \ar@{-}[ul]_>>>{d} \\ 	}\end{xy}\]
	The inverse operation of collapsing is called \textbf{expanding}.
	\begin{lemma}[{\cite[Lemma 19]{MosherSageevWhyte02}}] \label{lemma:collapse} Suppose $(\Bg,j)$ is obtained from $(\Ag,i)$ by collapsing the edge $e$. Then there is a natural injective copci $\varphi : \pi_1(\Ag,i) \to \pi_1(\Bg,j)$. If moreover $i(\overline e) = 1$ then $\varphi$ is an isomorphism.
	\end{lemma}
	If $(\Bg,j)$ is obtained from $(\Ag,i)$ by collapsing the edge $e$, then the preimage of $e$ in $\widetilde{(\Ag,i)}$ is a $\pi_1(\Ag,i)$-invariant disjoint union of finite subtrees, and  the universal cover $\widetilde{(\Bg,j)}$ is obtained from $\widetilde{(\Ag,i)}$ by contracting each such finite subtree to a point.

	\subsection{A key example}
		\label{section:key_example} For each $k \geq 0$ we define
			\[ (A,i_k) = \begin{xy}
			 \xymatrix{ 
			 \bullet \ar@{-}[r]^<{2}^>{2^k} & \bullet \ar@{-}@(ur,dr)^<<{1}^>>{2}
			 &} \end{xy} \text{ and } (B,j_k) = \begin{xy}
			 \xymatrix{ 
			 \bullet \ar@{-}[r]^<{2}^>{2^k} & \bullet \ar@{-}@/_/[r]_<{1}_>{1} \ar@{-}@/^/[r]^<{1}^>{2} & \bullet
			 &} \end{xy} \]
			 There is a sequence of collapses and expansions
			 \[ {\begin{xy}\xymatrix{ 
			 \bullet \ar@{-}@(ur,dr)^<{2}^>{4}
			 }\end{xy}} \phantom{aaa} \to (A,i_0) \to (B,i_0) \to (A,i_1) \to (B,i_1) \to \ldots \]
			 
			 In view of the isomorphism $\ZZ_2 \cong \ZZ_2 \ast_{2 \ZZ_2} 2 \ZZ_2$ there are isomorphisms of the corresponding procyclic groupings
		\[G_{2,4} = \pi_1(\ZZ_2 (\begin{xy}\xymatrix{ 
			 \bullet \ar@{-}@(ur,dr)^<{2}^>{4}
			 }
			 \end{xy}\ \ \ \ )) \cong \pi_1( \ZZ_2 (A,i_0)) \cong \pi_1( \ZZ_2 (B,i_0)) \cong \ldots \cong \pi_1( \ZZ_2 (A,i_k)) \cong \ldots\]

			 Observe that the universal covering tree $\widetilde{(A,i_k)}$ is a subdivision of the regular tree of degree $(2^k+2+1)$. Hence $G_{2,4}$ acts geometrically and minimally on regular trees of arbitrarily large degree.
	
	\subsection{Deformation spaces and their invariants} \label{sec:defn_space_invariants}
		We shall rely on the following special case of Forester's Deformation Theorem \cite{Forester02}, see also \cite{BassLubotzky94, MosherSageevWhyte02} for similar results.
		\begin{theorem} \label{theorem:deformation} If a l.c.\ group $G$ acts geometrically (without inversions) on two trees $T_1, T_2$ then the quotient edge-indexed graphs are related by a finite sequence of collapses and expansions.
		\end{theorem}
		Forester's result is much more general: there is no continuity nor properness assumption, and trees need not be locally finite. Instead, it is required that both actions yield the same elliptic subgroups. In our case, a subgroup $H < G$ is elliptic (for its action on either tree) if and only if it it has compact closure in $G$.
		\begin{definition} A finite edge-indexed graph $(\Ag,i)$ is called \textbf{minimal} if it  has no vertex with a single outgoing edge of index $1$. A vertex of $(\Ag,i)$ is \textbf{inessential} if it has exactly two outgoing edges $e,e'$ of index $1$ such that $\overline e \neq e'$. We further call $(\Ag,i)$ \textbf{essential} if it is minimal and has no inessential vertex. Given a finite edge-indexed graph $(\Ag,i)$ we denote by $(\Ag,i)_{\text{min}}$ the minimal edge-indexed graph obtained from $(\Ag,i)$ by removing successively all vertices with a single outgoing  edge of degree $1$. We denote by $(\Ag,i)_{\text{ess}}$ the essential edge-indexed graph obtained from $(\Ag,i)_{\text{min}}$ by removing successively all inessential vertices.
		\end{definition}
		
		\begin{remark} \label{remark:essential_deformation} If two finite edge-indexed graphs $(\Ag,i)$ and $(\Bg,j)$ are related by a finite sequence of collapses and expansions, then $(\Ag,i)_{\text{min}}$ and $(\Bg,j)_{\text{min}}$ (resp. $(\Ag,i)_{\text{ess}}$ and $(\Bg,j)_{\text{ess}}$) are related by a finite sequence of collapses and expansions between minimal (resp. essential) edge-indexed graphs. \hfill \qed
		\end{remark}
		\begin{remark} Let $G \in \Tree$, let $T$ be a geometric $G$-tree without inversions and let $(\Ag,i)$ be the quotient (finite connected) edge-indexed graph. Then the action of $G$ on $T$ is minimal (in the sense that it admits no proper invariant subtree) if and only if $T$ has no vertex of degree $1$ if and only if $(\Ag,i)$ is minimal. In particular, if two l.c.\ groups act geometrically on a tree $T$, then either both actions are minimal or none is.
		
			It is incorrect to describe $(\Ag,i)$ being essential in terms of $T$ having no vertex of degree $2$. Indeed, removing vertices of degree $2$ from a given $G$-tree may introduce edge-inversions. Being essential cannot be expressed solely in terms of $T$: it also depends on the action of $G$.
		\end{remark}
		
		Given a minimal geometric $G$-tree $T$ without inversions with quotient edge-indexed graph $(\Ag,i)$, define the set of \textbf{invariant primes} $P_G$ to be the set of primes dividing $i(e)$ for some $e \in E\Ag$. The notation is justified by the following:
		\begin{corollary} \label{corollary:invariant_primes} Given $G \in \Tree$ the set of invariant primes $P_G$ is independent of the minimal geometric $G$-tree chosen above, and hence is an invariant of $G$.
		\end{corollary}
		\begin{proof} In view of Theorem~\ref{theorem:deformation} and Remark~\ref{remark:essential_deformation} it is enough to show that if $(\Ag,i)$ and $(\Bg,j)$ are minimal edge-indexed graphs and $(\Bg,j)$ is obtained from $(\Ag,i)$ by collapsing the edge $e \in E\Ag$ then both edge-indexed graphs yield the same set of invariant primes. But since $(\Bg,j)$ is minimal there exists an edge $f \neq e \in E\Ag$ such that $\alpha(e) = \alpha(f)$ so that $j(f) = i(f) i(\overline e)$. The result now follows.
		\end{proof}
		
		\subsection{Consequences for minimal actions}
		
		\begin{lemma} \label{lemma:bound_vertex_degree_1} Let $T$ be a minimal geometric $G$-tree, let $\Af$ be the quotient graph of groups and let $(\Ag,i)$ be the quotient edge-indexed graph. Then the degree of any vertex of $\Ag$ is bounded by a constant depending only on $G$.
		\end{lemma}
		\begin{proof} First we observe that the kernel of the quotient map $G = \pi_1(\Af) \to \pi_1(\Ag)$ is the subgroup generated by elliptic subgroups. A subgroup of $G$ is elliptic if and only if it is compact, so that $\pi_1(\Ag) \cong G / \langle \cup \{K < G \text{ compact} \} \rangle$. In particular, the first Betti number of the quotient graph  $\Ag$ is independent of $T$.
		
		Next, since the action is minimal the vertex group of each vertex of degree $1$ of $\Ag$ is a maximal compact subgroup of $G$, and moreover the vertex groups of any two such distinct vertices are non-conjugate in $G$. Thus the number of vertices of degree $1$ of the graph $\Ag$ is bounded by the number of conjugacy classes of maximal compact subgroups of $G$, which is finite.
		
		Finally the degree of any vertex of the graph of $\Ag$ is bounded above by $V_1 + 2b_1$ where $V_1$ is the number of vertices of degree $1$ in $\Ag$ and $b_1$ is the first Betti number of $\Ag$ which as seen above only depend on $G$.
		\end{proof}
		Although not directly useful for our purposes, we note that the above lemma imposes rather strong conditions on the possible degrees of vertices of any minimal geometric $G$-tree:
		\begin{prop} Given a group $G \in \Tree$, there is a number $K$ such that for any minimal geometric $G$-tree $T$, the degree of any vertex is of the form \[\sum_{k \in \{1,\ldots,K_0\}} \prod_{p \in P_G} p^{j_k} \]
		for some $K_0 \leq K$ and for integers $j_k \in \NN$. \qed
		\end{prop}
		Using the terminology from Section~\ref{sec:primes_not_in_s_2_N} the conclusion can be rephrased as follows: there are $K$ and $N = \max P_G$ such that the degree of any vertex is in $\mathcal S_{K,N}$. 
		
		We show that unimodular groups enjoy much stronger restrictions.
		\begin{prop}\label{prop:bound_on_degree} Let $U\in \Tree$ be a unimodular group. Then there is a natural number $N_U$ such that if $T$ is a geometric minimal $U$-tree then all vertices of $T$ have degree at most $N_U$.
		\end{prop}
		\begin{proof} Although it is possible to give a direct proof, we shall use Bass-Kulkarni's result that $U$ contains a free discrete cocompact subgroup $\Gamma\cong F_n$. Fix a geometric minimal $U$-tree $T$. Since the action is cocompact and minimal it follows that $T$ has no vertex of degree $1$. Moreover $\Gamma$ acts freely and cocompactly on $T$, and the quotient (finite) graph $\Gamma \backslash T$ has fundamental group $\Gamma \cong F_n$ and no vertex of degree $1$. It now follows from an Euler characteristic argument that the degree of any vertex of $\Gamma \backslash T$ (and hence of $T$) is bounded by $N_U = 2n$.
		\end{proof}
		
		\begin{examples} The above proof shows that if $U=F_2$ acts minimally and geometrically on a tree $T$, then $T$ is a subdivision of the $3$ or $4$-regular tree. A slightly more careful argument shows that for $U=\PSL_2(\ZZ)$ any geometric minimal $U$-tree is a subdivision of the $3$-regular tree. 
		
		The conclusion of Proposition~\ref{prop:bound_on_degree} also holds for some non-unimodular groups. For example, one can show that for $p$ prime and 
		\[G=\QQ_p \rtimes \langle p \rangle  = \pi_1(\ZZ_p(\begin{xy}\xymatrix{ 
			 \bullet \ar@{-}@(ur,dr)^<{1}^>{p}
			 }
			 \end{xy}\ \ \ \ ))\]
			 any minimal, proper cocompact $G$-tree is a subdivision of the $p+1$-regular tree.
		\end{examples}
		
	\subsection{More moves}
	We introduce further moves in order to ease our arguments.
	\subsubsection{Edge blow-up}
	Let $(\Ag,i)$ be an edge-indexed graph, and let $e$ be an edge with $i(e) \neq 1$ and $i(\overline e) \neq 1$. We say that $(\Bg,j)$ is obtained from $(\Ag,i)$ by a \textbf{standard blow up} of the edge $e$ if the following holds : 
	\begin{itemize}
		\item $V\Bg = V\Ag$
		\item $E\Bg = E\Ag \sqcup \{e', \overline e'\}$
		\item $\alpha(e')  = \alpha(e)$ and $\alpha(\overline e') = \alpha(\overline e)$.
		\item If $f \in E\Ag \backslash \{e, \overline e\}$, define $j(f) = i(f)$. Further define \[j(e') = j(\overline e') = 1; j(e) = i(e)-1 \text{ and } j(\overline e) = i(\overline e) -1.\]
	\end{itemize}
	\[ \begin{xy}
	\xymatrix{ \bullet \ar@{-}[rr]^<{m}^>{n}^{\{e,\overline e\}} & &  \bullet & \Rightarrow & \bullet \ar@{-}@/^/[rr]^<<{1}^>>{1} \ar@{-}@/_/[rr]_<<<{m-1}_>>>{n-1} & & \bullet 	}\end{xy}\]
	It is clear that the natural map $p : (\Bg,j) \to (\Ag,i)$ identifying $e$ with $e'$ and $\overline e$ with $\overline e'$ is a covering map. 
	
	More generally, we say that $(\Bg,j)$ is obtained from $(\Ag,i)$ by a \textbf{blow up} of the edge $e$ if there is a covering map $p : (\Bg,j) \to (\Ag,i)$ such that $p$ restricts to an isomorphism $p : \Bg \backslash p^{-1}(\{e,\overline e\}) \to \Ag\backslash p^{-1} (\{e,\overline e\})$. In other words, we allow to blow up $e$ to possibly more than $2$ edges, and the indices may be arbitrary.
	
	The following lemma is a direct consequence of Lemma~\ref{lemma:covering}.
	\begin{lemma} \label{lemma:blow-up} Let $(\Ag,i)$ be a finite connected edge-indexed graph. If $(\Bg,j)$ is obtained from $(\Ag,i)$ by blowing up an edge, there is a injective copci $p^* : \pi_1(\Bg,j) \to \pi_1(\Ag,i)$.
	\end{lemma}

	\subsubsection{Sliding across an edge}
	Let $(\Ag,i)$ be an edge-indexed graph with two edges $e \neq f \in E\Ag$ such that $\alpha(e) = \alpha(f)$ and such that $i(e) = 1$. Then we define the edge-indexed graph obtained from $(\Ag,i)$ by \textbf{sliding the edge $f$ across $e$} by the following diagram:
	\[ \begin{xy}
	\xymatrix{ & \ar@{-}[dl]_>>{k}_f& & & & \ar@{-}[dr]^>>{kn}^f\\ 
	\bullet \ar@{-}[rr]_<<{1}_>>{n}_{\{e,\overline e\}} & &  \bullet & \Rightarrow & \bullet \ar@{-}[rr]_<<{1}_>>{n} & & \bullet 	}\end{xy}\]
	Remark that a sliding move is the composition of an expansion followed by a collapse.
	
	\subsubsection{Graph cover}
	
	Let $(\Ag,i)$ be a finite connected edge-indexed graph. We say that $(\Bg,j)$ is a \textbf{graph cover} of $(\Ag,i)$ if there is a map $p : (\Bg,j) \to (\Ag,i)$ which is both a covering map of edge-indexed graphs and a covering map of graphs $\Bg \to \Ag$. In other words, $p : \Bg \to \Ag$ is a surjective map such that
	\begin{itemize}
		\item for each vertex $v \in V\Bg$ the map $p$ induces a bijection $\alpha^{-1}(v) \to \alpha^{-1}(p(v))$
		\item for each  edge $e \in E\Bg$ one has $j(e) = i(p(e))$.
	\end{itemize}
	 
	\section{Connectedness and upper bounds} \label{sec:upper_bounds}
	
	This section is devoted to the proof of Theorem~\ref{theorem:upper_bounds}.
	
	\begin{prop} \label{prop:sequence_of_moves} Given any finite connected edge-indexed graph $(\Ag,i)$, then by doing a sequence of collapses, followed by a sequence of blow-ups, followed by a sequence of collapses, one can obtain an edge-indexed graph with a single vertex. 
	\end{prop}
	\begin{proof} After a finite sequence of collapses, we can ensure that no non-loop edge is indexed by $1$. Next we do a standard blow up on all non-loop edges. Now there exists a maximal tree of the graph with trivial indexing. Collapsing any edge in this tree does not change the indexing of any other edge, so that it is possible to collapse all edges in this maximal tree one after the other. The result is thus an edge-indexed graph with a single vertex.
	\end{proof}
	 
	We are now in position to prove the following:
	\begin{theorem}\label{theorem:commation_to_regular} Given any $G \in \Tree$ there is some $n \geq 3$ such that there is a commation $G \nearrow \nwarrow \nearrow \Aut(T_n)$.
	\end{theorem}
	\begin{proof}
		By Lemma~\ref{lemma:qfree_group_copci_elg} there is a copci $G \to \pi_1(\Ag,i)$ for some finite connected edge-indexed graph $(\Ag,i)$. We fix a sequence of moves provided by Proposition~\ref{prop:sequence_of_moves} $(\Ag,i) \stackrel{\text{collapses}}{\longrightarrow} (\Bg,j) \stackrel{\text{edge blowup}}{\longrightarrow} (\Cg,k) \stackrel{\text{maximal forest collapse}}{\longrightarrow} (\Dg,l)$ where $\Dg$ has a single vertex. In view of Lemmas~\ref{lemma:covering} and~\ref{lemma:collapse} it follows that there is a commation \[G \nearrow \pi_1(\Ag,i) \nearrow \pi_1(\Bg,j) \nwarrow \pi_1(\Cg,k) = \pi_1(\Dg,l) \nearrow \Aut(T_n)\] with $n = \sum_{e\in E\Dg} l(e)$.
	\end{proof}
	
	Recall that the automorphism group of any regular tree is unimodular, and that for any two unimodular groups $U, U' \in \Tree$ there is a commation $U \nwarrow \nearrow U'$ by Bass-Kulkarni's Theorem, so that we obtain
	\begin{corollary}\label{corollary:commation_1_unimod} Given any group $G \in \Tree$ and any unimodular group $U \in \Tree$ there is a commation $G \nearrow \nwarrow \nearrow \Aut(T_n) \nwarrow \nearrow U$.\hfill {\qed}
	\end{corollary}
	
	\begin{remark}\label{remark:regular_degree} Suppose the edge-indexed graph $(\Bg,j)$ contains a trivially indexed maximal forest. Then collapsing this maximal forest results in an edge-indexed graph whose universal cover is the $n$-regular tree (and hence there is a copci $\pi_1(\Bg,j) \nearrow \Aut(T_n)$) where
	\[n = \sum_{e \in E\Bg} j(e) - |V\Bg| + 1.\]
	\end{remark}
	
	\begin{theorem}\label{theorem:upper_bound_diam} Given any $G, G' \in \Tree$ there is some $n \geq 3$ such that there is a commation $G \nearrow \nwarrow \nearrow \Aut(T_{n}) \nwarrow \nearrow \nwarrow G'$.
	\end{theorem}
	\begin{proof} We start as in the proof of Theorem~\ref{theorem:commation_to_regular} by finding copcis
	\[G \nearrow \pi_1(\Bg,j)  \text{ and } \pi_1(\Bg',j') \nwarrow G'  \]
	where all non-loop edges of $(\Bg,j)$ and $(\Bg',j')$ have index at least $2$. If either $\pi_1(\Bg,j)$ or $\pi_1(\Bg',j')$ is unimodular the result follows from Corollary~\ref{corollary:commation_1_unimod}, so we assume from now on that both $\pi_1(\Bg,j)$ and $\pi_1(\Bg',j')$ are non-unimodular. In particular, neither $\Bg$ nor $\Bg'$ are trees.
	
	We define \[m = \sum_{e \in E\Bg} j(e) - |V\Bg| \text{ and } m' = \sum_{e \in E\Bg'} j'(e) - |V\Bg'|.\] Since the graphs $\Bg$ and $\Bg'$ are not trees we can find connected covers $\tilde \Bg$ and $\tilde \Bg'$ of $\Bg$ and $\Bg'$ respectively, which are $m'$ and $m$-sheeted respectively. Let $\tilde j$ (resp. $\tilde j'$) be the indexing of $\tilde \Bg$ (resp. $\tilde \Bg'$) obtained by lifting the indexing $j$ of $\Bg$ (resp. $j'$ of $\Bg'$).
	Let $(\Cg,k)$ and $(\Cg',k')$ be obtained from $(\tilde \Bg,\tilde j)$ and $(\tilde \Bg',\tilde j')$ respectively by doing a standard blow up on all non-loop edges. We have
	\[ \sum_{e \in E\Cg} k(e) - |V\Cg| + 1 = \sum_{e \in E\tilde \Bg} \tilde j(e) - |V \tilde \Bg| + 1  = m' \left(\sum_{e \in E\Bg} j(e) - |V\Bg|\right) + 1 = m' m + 1\]
	and similarly 
	\[ \sum_{e \in E\Cg'} k'(e) - |V\Cg'| + 1 = \sum_{e \in E\tilde \Bg'} \tilde j'(e) - |V \tilde \Bg'| + 1  = m \left(\sum_{e \in E\Bg'} j'(e) - |V\Bg'|\right) + 1 = m m' + 1.\]
	In view of Remark~\ref{remark:regular_degree} collapsing trivially indexed maximal forests in $(\Cg,k)$ and $(\Cg',k')$ results in copcis $(\Cg,k) \nearrow \Aut(T_{mm' + 1}) \nwarrow (\Cg',k')$. Thus letting $n = mm'+1$ we have obtained the desired commation
	\[ G \nearrow \pi_1(\Bg,j)  \nwarrow \pi_1(\Cg,k) \nearrow \Aut(T_n) \nwarrow \pi_1(\Cg',k') \nearrow \pi_1(\Bg',j') \nwarrow G'.\qedhere \]
	\end{proof}
	
	The following upper bound is slightly more delicate to obtain. 
	\begin{theorem} \label{theorem:upper_bound_radius}
		Given any $G \in \Tree$ there is a commation $G \nearrow \nwarrow \nearrow \nwarrow G_{2,4}$.
	\end{theorem}
	\begin{proof} Recall from Section~\ref{section:key_example} that $G_{2,4}$ is a locally compact group which acts geometrically on the regular tree of degree $2^k + 2 + 1$ for any integer $k \geq 0$. 
	
	Again we start by taking a copci
	\[G \nearrow \pi_1(\Bg,j) \]
	where all non-loop edges of $(\Bg,j)$ have index at least $2$. After subdividing loop edges and collapsing bigons in a subcover, we can make sure that $(\Bg,j)$ has no loop edge, so that all edges have index at least $2$.
	If $\pi_1(\Bg,j)$ is unimodular then the conclusion follows from Corollary~\ref{corollary:commation_1_unimod}, so we assume from now on that $\pi_1(\Bg,j)$ is non-unimodular. In particular, $\Bg$ is not a tree, has at least $2$ geometric edges, and not all non-separating edges have index $2$.
	
	We let $m = \sum_{e \in E\Bg} j(e) - |V\Bg|$. Let $i_0$ be an integer large enough so that $2^{i_0} + 2 \geq m^2$ and let $i_1 \geq m$ and $0 \leq i_2 \leq m$ be integers such that \[i_1 m + i_2 = 2^{i_0} + 2.\]
	Consider an $i_1$-sheeted cover $\tilde \Bg$ of $\Bg$ and let $\tilde j$ be the lift of the indexing $j$. Let $e_0$ be a non-separating edge of $\tilde \Bg$ such that $\tilde j(e_0) \geq 3$. Let also $e_1, \ldots, e_{i_2}$ be edges such that $\{e_i, \overline{e_i}\}\neq \{e_{i'}, \overline{e_{i'}}\}$ for any $0 \leq i \neq i' \leq i_2$. Note that such edges exist since $\Bg$ has at least $2$ geometric edges and $\tilde \Bg$ is an $i_1$-sheeted cover of $\Bg$ with $i_1 \geq m \geq i_2$. Let $(\Cg,k)$ be obtained from $(\tilde \Bg,\tilde j)$ by the following moves:
	\begin{enumerate}
		\item blow up the edge $\{e_0,\overline{e_0}\}$ to two geometric edges indexed $2 - 1$ and $(\tilde j(e_0)-2) - (\tilde j(\overline{e_0})-1)$ respectively.
		\item for each $1 \leq i \leq i_2$, blow up the edge $\{e_i,\overline{e_i}\}$ as follows:
		\begin{enumerate}
			\item if $\tilde j(e_i) = 2$ or $\tilde j(\overline{e_i}) = 2$ then blow up the edge to a $1 - 1$ edge and a $1 - n$ edge.
			\item otherwise, blow-up the edge to three edges, two of them with trivial indexing.
		\end{enumerate}
		\item blow up all other geometric edges $\{e, \overline e\}$ to two edges, one of which has trivial indexing.
	\end{enumerate}
	This ensures that the following holds for $(\Cg,k)$: 
	\begin{itemize}
		\item There exists a maximal subtree $T$ of $\Cg$ with trivial indexing. (This uses the fact that $e_0$ is non-separating)
		\item There is an edge $f_0$ outside $T$ such with $k(f_0) = 1$ and $k(\overline{f_0}) = 2$.
		\item There are pairwise distinct edges $f_i$ outside $T \cup \{f_0, \overline{f_0}\}$ for $i \in \{1, \ldots, i_2\}$ such that $k(f_i) = 1$.
	\end{itemize}
		
	Now for each $i \in \{1, \ldots, i_2\}$ slide the edge $f_i$ along the maximal tree up to $\alpha(f_0)$ and then across $f_0$, producing a graph $(\Dg,l)$. Thus the only change in the indices is that the resulting edges $f_i'$ have index $l(f_i') = 2$ whereas $k(f_i) = 1$. Thus we have 
	\begin{align*} \sum_{e \in E\Dg} l(e) - |V\Dg| + 1 & = \sum_{e \in E\Cg} k(e) + i_2 - |V\Cg| + 1 \\
	& = \sum_{e \in E\tilde \Bg} \tilde j(e) - |V \tilde \Bg| + i_2 + 1 \\
	& = i_1 \left(\sum_{e \in E\Bg} j(e) - |V\Bg|\right) + i_2 + 1 \\
	& = i_1 m + i_2 + 1 = 2^{i_0}+2+1.
	\end{align*}
	Moreover $(\Dg,l)$ still has a trivially indexed maximal tree so that by Remark~\ref{remark:regular_degree} we have a copci $\pi_1(\Dg,l) \nearrow \Aut(T_{2^{i_0}+2+1})$. Thus we have the desired commation
	\[ G \nearrow \pi_1(\Bg,j) \nwarrow \pi_1(\tilde \Bg,\tilde j) \nwarrow \pi_1(\Cg,k) \nearrow \pi_1(\Dg,l) \nearrow \Aut(T_{2^{i_0}+2+1}) \nwarrow G_{2,4}.\qedhere \]
	\end{proof}
	\section{Triangles of procyclic groups} \label{sec:key_examples}
		Fix pairwise distinct primes $q,r,s$. We fix the following edge-indexed graph $(A,i)$ with three vertices $v_1, v_2, v_3$ and $6$ edges $e_1, e_2, e_3, \overline{e_1}, \overline{e_2}, \overline{e_3}$
			\[(\Ag,i) =\begin{xy}
			 \xymatrix{ v_2 \ar@{-}|@{<}[rr]_{e_1}^<<<{2}^>>>{q} & & v_1 \ar@{-}|@{<}[dl]^<<<{2}^>>>{s}_{e_3} \\ 
			  & {v_3} \ar@{-}|@{<}[lu]^<<<{2}^>>>{r}_{e_2} &
			 }\end{xy}.\]
			 
		Let $n = 2 \times q \times r \times s$ and $\ZZ_n = \ZZ_2 \times \ZZ_q \times \ZZ_r \times \ZZ_s$. Let $\Af$ be the graph of groups with underlying graph $\Ag$ as above, with all edge and vertex groups isomorphic to $\ZZ_n$ and such that the edge monomorphism $\alpha_e : \Af_e \to \Af_{\alpha(e)}$ corresponds to multiplication by $i(e)$. We define $H^{2,2,2}_{q,r,s} = \pi_1(\Af)$.
		
		\begin{definition}\label{definition:S_2,N} Given $N\geq 1$ we define the sets of natural numbers \[\Pi_N = \left\{\prod_{n = 1}^N n^{k_n} \mid k_n \in \NN \text{ for each } n\right\}\] and \[\mathcal S_{2,N} = \Pi_N \cup (\Pi_N + \Pi_N).\]
		\end{definition}
		This section is devoted to the proof of the following result.
		\begin{theorem} \label{theorem:key_example}Let ${q,r,s} \geq 11$ be pairwise distinct primes.
		\begin{enumerate} 
			\item \label{theorem:key_example_minimal} Suppose there is a commation $H^{2,2,2}_{q,r,s} \nwarrow \nearrow \nwarrow G$ then there is a commation $H^{2,2,2}_{q,r,s} \nearrow \Aut(T^{2,2,2}_{q,r,s}) \nwarrow G$. 
			\item \label{theorem:key_example_rigidity} Suppose moreover that $q,r,s$ are not in $\mathcal S_{2,N}$. Then $G$ cannot act geometrically on a tree with degree bounded above by $N$.
		\end{enumerate}
	\end{theorem}
		
		We postpone until the end of this section the (technical) proof of the first assertion, which is based on a precise description of all cocompact subgroups of $H^{2,2,2}_{q,r,s}$. Instead, we start by proving the second assertion assuming the first one.
		\begin{proof}[Proof of \eqref{theorem:key_example_rigidity} assuming \eqref{theorem:key_example_minimal}] Observe that it is enough to show that the set of invariant primes $P_G$ of $G$ is not contained in the interval $[1, \ldots, N]$. This indeed prohibits $G$ from acting geometrically on any tree with degree bounded above by $N$ in view of Corollary~\ref{corollary:invariant_primes}.
		
		Seeking a contradiction, assume that all invariant primes of $G$ are at most $N$. 
			 Recall that by definition $T^{2,2,2}_{q,r,s}$ is the universal covering tree $\widetilde{(\Ag,i)}$. Consider the action of $G$ on $T^{2,2,2}_{q,r,s}$ given by Theorem~\ref{theorem:key_example}\eqref{theorem:key_example_minimal} and let $(\Bg,j)$ be the quotient edge-indexed graph. Thus there is a covering of edge-indexed graphs $\pi : (\Bg,j) \to (\Ag,i)$. Since $q$ is not in $\mathcal S_{2,N}$ it follows that for any vertex $v \in \pi^{-1}(v_1)$ there are at least three lifts of $e_1$ starting at $v$ (and similary for $r$ and $s$) so that we have \[ |\pi^{-1}(e_i)| \geq 3 |\pi^{-1}(v_i)| \] for each $i \in \{1,2,3\}$. Moreover, since $i(\overline{e_i})=2$ for each $i \in \{1,2,3\}$ there can be at most $2$ edges above $\overline{e_i}$ starting from a given vertex of $\Bg$. This in turn shows that \[ |\pi^{-1}(\overline{e_i})| \leq 2 |\pi^{-1}(v_{i+1})|\] where indices are considered modulo $3$.
			 Combining the two relations and proceeding cyclically we obtain
			 \[ |\pi^{-1}(v_1)| \leq \frac{1}{3} |\pi^{-1}(e_1)| \leq \frac{2}{3} |\pi^{-1}(v_2)| \leq \frac{2}{3^2} |\pi^{-1}(e_2)| \leq \frac{2^2}{3^2} |\pi^{-1}(v_3)| \leq \frac{2^2}{3^3} |\pi^{-1}(e_3)| \leq \frac{2^3}{3^3} |\pi^{-1}(v_1)| \]
			 But $\pi^{-1}(v_1)$ is non-empty so we must have $(2/3)^3 \geq 1$ which is a contradiction.
		\end{proof}
		
		\subsection{Closed cocompact subgroups of $H^{2,2,2}_{q,r,s}$}
		
		\begin{notation} Throughout the rest of this section, we fix pairwise distinct primes $q,r,s \geq 11$, fix $n = 2\times q \times r \times s$ and rename $H^* = H^{2,2,2}_{q,r,s}$ for better readability.
		\end{notation}

		\begin{prop} \label{prop:cocpt_subgp_fi} Any closed cocompact subgroup $H \leq H^* = H^{2,2,2}_{q,r,s}$ contains the kernel of the map $H^* \twoheadrightarrow \pi_1(\Ag) \cong \ZZ$. Thus $H$ is normal of finite index, and $H \cong \pi_1(\ZZ_n(\tilde \Ag,\tilde i))$ for some finite (graph) cover $\tilde \Ag$ of $\Ag$ (and $\tilde i$ is the lift of the indexing $i$).
		\end{prop}
		The proof requires some setup. In order to stay as self-contained as possible, we avoid using the (rather technical) language of coverings of graphs of groups, see Remark~\ref{rmk:coverings_gog}. Instead, we extract the needed observations which are then turned into combinatorics through the use of Lemma~\ref{lemma:combinatorics}. 
		
		We start by recalling elementary facts about procyclic groups
		\begin{lemma} \label{lemma:procyclic} Any closed subgroup $H_0 < \ZZ_n = \prod_{p \in \{2,q,r,s\}} \ZZ_p$ can be written in a unique fashion as \[ H_0 = \prod_{p \in \{2,q,r,s\}} p^{k_p} \ZZ_p \] where $k_p \in \NN \cup \{\infty\}$, with the convention $p^\infty = 0$.
		\end{lemma} 
		
		Let $T$ be the (locally finite) Bass-Serre tree of $\Af$ (= the universal covering tree of $(\Ag,i)$). Consider the action of $H$ on $T$. Since $H$ is cocompact in $H^*$ the action of $H$ on $T$ is cocompact, i.e. the quotient graph $\Bg = H \backslash T$ is finite. Let $\pi : \Bg \to \Ag$ denote the map between quotients induced by the inclusion $H < H^*$. For every vertex or edge $x \in VT \cup ET$, the stabilizer $H^*_x$ is isomorphic to $\ZZ_n$. For each $p \in \{2,q,r,s\}$ we let $k_p^x \in \NN \cup \{\infty\}$ be such that the the expression for the closed subgroup $H_x < H^*_x$ according to Lemma~\ref{lemma:procyclic} is \[H_x = \prod_{p \in \{2,q,r,s\}} p^{k_p^x} \ZZ_p < \ZZ_n \cong H^*_x.\] Moreover, the number $k_p^x$ is constant on each $H$-orbit, so that we can unambiguously define $k_p^x$ for $x \in V\Bg \cup E\Bg$. 
		
		\begin{remark} \label{rmk:coverings_gog} The situation translates into the language of coverings of graphs of group (in the sense of Bass \cite{Bass93}) as follows : we have a covering of graphs of groups $\pi:\Bf \to \Af$ 
		and for every vertex or edge $x \in V\Bg \cup E\Bg$, the tuple $(k_p^x)_{p\in \{2,q,r,s\}}$ is so that $\pi(\Bf_x) = \prod_{p \in \{2,q,r,s\}} p^{k_p^x} \ZZ_p < \ZZ_n \cong \Af_{\pi(x)}$. 
		\end{remark}		
		
		The following lemma provides necessary conditions the graph $V\Bg$ and the numbers $k_p^x$ must satisfy if they come from the above construction.
		\begin{lemma} \label{lemma:combinatorics} Let $\pi : \Bg \to \Ag$ and $k_p^x$ be as above.
		
		Fix a vertex $v \in V\Bg$ and an edge $e \in EA$ originating at $\pi(v)$ with index $p = i(e) \in \{2,q,r,s\}$. Then
		\begin{enumerate} 
			\item If $k_{p}^v = 0$ then there is one lift $\tilde e$ of $e$ starting at $v$. We have $k_p^{\tilde e} = 0$ and $k_{p'}^{\tilde e} = k_{p'}^v$ for any prime $p' \in \{2,q,r,s\}$ with $p \neq p'$. \label{item:1lift}
			\item If $k_{p}^v > 0$ then there are $p$ lifts of $e$ starting at $v$. For any such lift $\tilde e$ we have $k_p^{\tilde e} = k_p^v - 1$ and $k_{p'}^{\tilde e} = k_{p'}^v$ for any prime $p' \in \{2,q,r,s\}$ with $p \neq p'$. (with the convention that $\infty - 1 = \infty$) \label{item:plifts}
		\end{enumerate}
		\end{lemma}
		\begin{proof} Let $w$ be a lift of $v$ in $T$ and let $f$ be a lift of $e$ in $T$ with $\alpha(f) = w$. In particular for any $p_0 \in \{2,q,r,s\}$ we have $k_{p_0}^v = k_{p_0}^w$ and letting $\tilde e$ denote the projection of $f$ in $\Bg$ we have $k_{p_0}^{\tilde e} = k_{p_0}^f$. Recall from the definition of $H^*$ that we can identify 
		\[H^*_w \cong \ZZ_n =\ZZ_p \times \prod_{p \neq p' \in \{2,q,r,s\}} \ZZ_{p'}.\] 
		Under this identification we have 
		\[H^*_f = p \ZZ_n = p\ZZ_p \times \prod_{p \neq p' \in \{2,q,r,s\}} \ZZ_{p'}\] 
		and 
		\[H_w = p^{k_p^w} \ZZ_p \times  \prod_{p \neq p' \in \{2,q,r,s\}} {p'}^{k_{p'}^w} \ZZ_{p'}.\]
		In particular we have $[H^*_w:H^*_f]=p$ and there are $p$ edges in the $H^*_w$-orbit of $f$. Moreover we have 
		\[H_f = H_w \cap H^*_f = p^{\max\{1,k_p^w\}} \ZZ_p \times  \prod_{p \neq p' \in \{2,q,r,s\}} {p'}^{k_{p'}^w} \ZZ_{p'}.\]
		
		The two cases in Lemma~\ref{lemma:combinatorics} shall appear by splitting into two cases, according to whether $H_w \subset H^*_f$ or not. 
		
		\eqref{item:1lift} Suppose that $H_w \not \subset H^*_f$. 
		It follows that $k_p^w = 0$ and $H_f = H_w \cap H^*_f$ has index $p$ in $H_w$. Thus there are $p$ edges in the $H_w$-orbit of $f$ so that $H_w f = H^*_w f$ and there is exactly one lift $\tilde e$ of $e$ starting at $v$ in $\Bg$. Finally, comparing the expressions for $H^*_f$ and $H_f$, we have $k_p^f = 0$ and $k_{p'}^f = k_{p'}^w$ for each $p' \in \{2,q,r,s\}$ with $p' \neq p$.
		
		\eqref{item:plifts} Suppose that $H_w \subset H^*_f = p\ZZ_n$. This means that $k_p^w >0$. We also have $H_f = H_w \cap H^*_f = H_w$ so that $H_w$ fixes the $p$ edges of the form $H^*_w f$ and hence there are exactly $p$ lifts of $e$ starting at $v$ in $\Bg$. Furthermore, we have a chain $H_f = H_w \subset H^*_f = p \ZZ_n \subset \ZZ_n = H^*_w$ so that $k_p^f = k_p^w - 1$ and for each prime $p' \neq p$ we have $k_{p'}^f = k_{p'}^w$. 
		\end{proof}
		
		Since the graph $\Ag$ is a cycle, it has two natural orientations (clockwise and counterclockwise). We say a path $\gamma$ in $\Bg$ runs \textbf{clockwise} or \textbf{counterclockwise} if its projection $\pi(\gamma)$ does. We record the following consequence of Lemma~\ref{lemma:combinatorics}:
		\begin{lemma}\label{lemma:cycle} Given any clockwise (or counterclockwise) cycle $\gamma$ in $\Bg$ and any prime $p\in \{2,q,r,s\}$ such that all $k_p^v$ are finite, then $k_p^v=0$ for all $v \in V\gamma$.
		\end{lemma}
		\begin{proof}
		Suppose without loss of generality that $\gamma$ is oriented clockwise. Write $\gamma = f_1, \ldots, f_m$ and for each $1 \leq l \leq m$ write $v_l = \alpha(f_{l+1}) = \omega(f_l)$ (where indices are considered modulo $m$). In view of Lemma~\ref{lemma:combinatorics} we have the following inequalities
		\[ k_2^{v_0} \leq k_2^{f_1} = k_2^{v_1} \leq \ldots \leq k_2^{f_m} = k_2^{v_m} = k_2^{v_0}\]
		and either all are infinite, or they are all $0$.
		For $p \in \{q,r,s\}$ we have similarly
		\[ k_p^{v_0} = k_p^{f_1} \geq k_p^{v_1} = \ldots = k_p^{f_m} \geq k_p^{v_m} = k_p^{v_0}\]
		and either all are infinite, or they are all $0$.
		\end{proof}		
		
		\begin{proof}[Proof of Proposition~\ref{prop:cocpt_subgp_fi}] Our goal is to show that $(B,j)$ is a (finite) graph cover of $(A,i)$. In view of Lemma~\ref{lemma:combinatorics} it is equivalent to show that $k_p^v = 0$ for each $v \in VB$ and each $p \in \{2,q,r,s\}$. Seeking a contradiction, we assume from now on that there exists some $v \in V\Bg$ and some $p \in \{2,q,r,s\}$ such that $k_p^v > 0$.
		
		\begin{step} There is some vertex $v$ and some $p \in \{q,r,s\}$ with $k_p^v > 0$.
		\end{step}
		Suppose on the contrary that $k_p^v = 0$ for each vertex $v \in V\Bg$ and each $p \in \{q,r,s\}$. Thus there is some vertex $v \in V\Bg$ with $k_2^v > 0$. Rephrasing using Lemma~\ref{lemma:combinatorics}, from each vertex there is a single edge in the counterclockwise direction, and there is a vertex with two edges in the clockwise direction. But this means that \[|\pi^{-1}(v_1)| = |\pi^{-1}(e_1)| \geq |\pi^{-1}(v_2)| = |\pi^{-1}(e_2)| \geq |\pi^{-1}(v_3)| = |\pi^{-1}(e_3) |\geq |\pi^{-1}(v_1)|\] with at least one of the inequalities being strict, which is impossible.
		
		\begin{step} There is some vertex $v$ and some $p \in \{q,r,s\}$ with $k_p^v = \infty$.
		\end{step}
		Let $v$ and $p \in \{q,r,s\}$ be such that $k_p^v > 0$ (which exist by the previous step). Start a clockwise path from $v$ until it eventually runs into itself at a vertex $v'$ (which it must since the graph is finite). Two situations may arise. Either $v = v'$ in which case $k_p^{v'} > 0$ by assumption. Otherwise observe that $v'$ is the starting point of two counterclockwise edges, from which it follows that $k_{p'}^{v'} > 0$ for some $p' \in \{q,r,s\}$. In either case $v'$ is on a clockwise cycle and has $k_{p''}^{v'} > 0$ for some $p'' \in \{q,r,s\}$. Thus $k_{p''}^{v'} = \infty$ by Lemma~\ref{lemma:cycle}.
		
		\begin{step} There is some $p \in \{q,r,s\}$ such that $k_p^v = \infty$ for all $v \in V\Bg$. 
		\end{step}
		It follows from from Lemma~\ref{lemma:combinatorics} that for any edge $e$ one has \[k_p^e = \infty \Leftrightarrow k_p^{\alpha(e)} = \infty.\] Thus since the graph $\Bg$ is connected, if there is a vertex $v$ and a prime $p \in \{2,q,r,s\}$ such that $k_p^v = \infty$, then $k_{p}^{v'} = \infty$ for all $v' \in V\Bg$. This step now follows from the previous one. 		
				
		\begin{step} Final contradiction
		\end{step}
		The conclusion of the previous step implies that there is some $m \in \{1,2,3\}$ such that from each vertex in $\pi^{-1}(v_m)$ there are exactly $p$ outgoing counterclockwise edges. Let $v \in \pi^{-1}(v_m)$.
	For each $l \in \NN$ let $V_l$ denote the set of vertices $w$ such that there is a counterclockwise path of lenth (exactly) $l$ from $v$ to $w$, and let $E_l$ be the set of edges from $V_l$ to $V_{l+1}$. Thus $V_0 = \{v\}$ and for each $l \in \NN$ we have $V_{3l} \subset \pi^{-1}(v_m)$. Thus it follows from the previous step that for each $l \in \NN$ there is some $p \in \{q,r,s\}$ such that $|E_{3l}| = p|V_{3l}|$. Since all $q,r,s$ are at least $11$, it follows that 
	\begin{equation}\label{eqn:X+}  11|V_{3l}| \leq |E_{3l}|\text{ for all } l \in \NN.
	\end{equation}
	Moreover, from every vertex starts at least one counterclockwise edge, hence
	\begin{equation}\label{eqn:YZ}  |V_l| \leq |E_l| \text{ for all } l \in \NN.
	\end{equation}
	Now recall that from any vertex originate two edges in the clockwise direction, so that 
	\begin{equation}\label{eqn:edges} |E_l| \leq 2|V_{l+1}| \text{ for all } l \in \NN.
	\end{equation}
	Combining \eqref{eqn:X+},~\eqref{eqn:YZ} and~\eqref{eqn:edges} we have for each $l \in \NN$
	\[ 11 |V_{3l}| \leq |E_{3l}| \leq 2 |V_{3l+1}| \leq 2|E_{3l+1}| \leq 4 |V_{3l+2}| \leq 4 |E_{3l+2}| \leq 8 |V_{3l+3}|.\]
	Thus we have $|V_0| = 1$ and $|V_{3l+3}| \geq \frac{11}{8} |V_{3l}|$ so that $\lim_{l \to \infty} |V_{3l}| = \infty$ in contradiction with the finiteness of $\Bg$. 
		\end{proof}
		\begin{prop} \label{prop:rigidity_cocpt_subgroup} Let $H < H^{2,2,2}_{q,r,s}$ be a closed cocompact subgroup. Let $T$ be a geometric $H$-tree. Then there is a copci $\Aut(T) \to \Aut(T^{2,2,2}_{q,r,s})$.
		\end{prop}
		\begin{proof} It will be convenient to rename $p_0 = q, p_1 = r, p_2 = s$ and consider indices modulo 3. Since $H$ is a cocompact subgroup of $H^{2,2,2}_{q,r,s}$ then by Proposition~\ref{prop:cocpt_subgp_fi} the edge indexed graph $(\tilde \Ag,\tilde i)$ corresponding to the action of $H$ on $T^{2,2,2}_{q,r,s}$ is cycle of length $3m$ for some natural number $m \geq 1$ and with edges indexed by $2,2,\ldots,2$ in one direction and by $p_0,p_1,\ldots, p_{3m-1}$ in the other. Without loss of generality, we can assume that $T$ has no vertex of degree $1$. Let $(\Bg,j)$ be the edge-indexed graph corresponding to the quotient $H \backslash T$. By Theorem~\ref{theorem:deformation} and Remark~\ref{remark:essential_deformation} the edge-indexed graph $(\Bg,j)$ is obtained from $(\tilde \Ag, \tilde i)$ by a sequence of collapses and expansions between essential edge-indexed graphs. Since $2,q,r,s$ are prime, one observes that any graph in the sequence is obtained from $(\tilde \Ag, \tilde i)$ by expanding a subset of the vertices to $1-1$ edges like in Figure~\ref{fig:blowup_vertex}.
		\begin{figure}[ht]
			\begin{center} \[ \begin{xy} \xymatrix{  & &  \ar@{-}[ld]^>{p_l}^<{2} \\ & \bullet \ar@{-}[lu]^<{2}^>{p_{l-1}} }\end{xy} 
		 \Rightarrow 
		 \begin{xy}  \xymatrix{  & & &  \ar@{-}[ld]^>{p_l}^<{2} \\ & \bullet \ar@{-}[lu]^<{2}^>{p_{l-1}}  & \bullet \ar@{-}[l]^<{1}^>{1}  }\end{xy} \]
		 	\end{center}
			\caption{Expanding a vertex of $(\tilde \Ag, \tilde i)$}
			\label{fig:blowup_vertex}
		\end{figure}
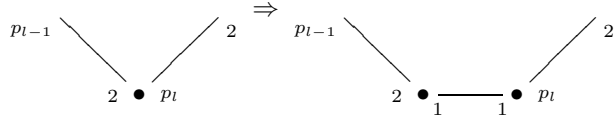
		
		 The graph $(\Bg,j)$ locally looks like the left or the right side of Figure~\ref{fig:blowup_vertex}. Thus any vertex of $T$ has degree $q + 1, r+1, s+1, q+2, r+2, s+2$ or $3$. Remark that these integers are all pairwise unequal since $q,r,s$ are pairwise distinct odd primes. We color the vertices of $T$ with the colors $\{0,1,2\}$ according to the following rule: the vertex $v$ is of color $l \in \{0,1,2\}$ if one of the following happens:
		\begin{enumerate}
		 	\item $v$ has degree $p_l +1$ or  $p_l +2$
			\item $v$ has degree $3$ and is adjacent to exactly one vertex of degree $p_l + 1$.
		\end{enumerate}
		Remark that the above rule detemines $l$ uniquely; any monochromatic subtree of $T$ either has $1$ or $2$ vertices; colors are $\Aut(T)$-invariant. Let $T'$ be the tree obtained by collapsing monochromic subtrees to (colored) vertices. Observe finally that any vertex of color $l \in \{0,1,2\}$ is adjacent to $p_l$ vertices of color $l+1$, to $2$ vertices of color $l-1$ and to no other vertex of color $l$. Thus in view of Remark~\ref{remark:recognition} the tree $T'$ is isomorphic to $T^{2,2,2}_{q,r,s}$ and the collapse defines a copci $\Aut(T) \to \Aut(T') = \Aut(T^{2,2,2}_{q,r,s})$.
		\end{proof}
		
		\begin{proof}[Proof of Theorem~\ref{theorem:key_example}\eqref{theorem:key_example_minimal}] Let $H^{2,2,2}_{q,r,s} \nwarrow H \nearrow G' \nwarrow G$ be a commation. Let $T$ be a geometric $G'$-tree. Then $H$ also acts geometrically on $T$, and we have a commation $H^{2,2,2}_{q,r,s} \nwarrow H \nearrow \Aut(T) \nwarrow G$.  Recall that the kernel of the action of $H$ on $T$ equals the unique maximal compact normal subgroup $\mathsf{W}(H)$ of $H$. 		Hence up to replacing $H$ by $H / \mathsf{W}(H)$ we can assume $H$ embeds as a cocompact subgroup of $H^{2,2,2}_{q,r,s}$. Applying Proposition~\ref{prop:rigidity_cocpt_subgroup} we see there is a commation \[H^{2,2,2}_{q,r,s} \nearrow \Aut(T^{2,2,2}_{q,r,s}) \nwarrow \Aut(T) \nwarrow G\] as desired.
		\end{proof}
		
		\subsection{Existence of primes not in $\mathcal S_{2,N}$} \label{sec:primes_not_in_s_2_N}
	In order to make sure that for every $N$ the statement of Theorem~\ref{theorem:key_example}\eqref{theorem:key_example_rigidity} is not empty, we establish the existence of infinitely many primes not in $\mathcal S_{2,N}$.
	
	Given natural numbers $N\geq 1$ and $d\geq 1$ we define \[\Pi_N = \left\{\prod_{n = 1}^N n^{k_n} \mid k_n \in \NN \text{ for each } n\right\};\]\[ \Sigma_{d,N} = \left\{ m_1 + \ldots + m_d \mid m_i \in \Pi_N\right\}\] and \[ \mathcal S_{d,N} = \bigcup_{i=1}^d \Sigma_{i,N}.\] 	\begin{prop} \label{prop:finite_series} For any natural numbers $N,d \geq 1$ 
	\[ \sum_{m \in \Sigma_{d,N}} \frac{1}{m} < + \infty.\]
	\end{prop}
	\begin{proof}
		We first note that by the Prime Factorization Theorem,
		\[ \Pi_N = \left\{\prod_{n = 1}^s p_s^{k_s} \mid k_s \in \NN \text{ for each } s\right\}\] where $\{p_1, \ldots, p_s\}$ is the set of primes in the interval $[1, N]$.
		
		By the Arithmetic Mean-Geometric Mean Inequality,
		\begin{align*} 
			\sum_{m \in \Sigma_{d,N}} \frac{1}{m} 
				&\leq \sum_{m_1 \in \Pi_N} \cdots \sum_{m_d \in \Pi_N} \frac{1}{m_1 + \cdots + m_d}\\
				& \leq \frac{1}{d} \sum_{m_1 \in \Pi_N} \cdots \sum_{m_d \in \Pi_N} \frac{1}{(m_1 \cdots m_d)^{1/d}} \\
				& \leq \frac{1}{d}\left( \sum_{m_1 \in \Pi_N} \frac{1}{m_1^{1/d}} \right)^d.
		\end{align*}
		Finally, since $p_i^{1/d} > 1$ it follows that
		\begin{align*}  \sum_{m_1 \in \Pi_N} \frac{1}{m_1^{1/d}} &= \sum_{k_1=0}^\infty \cdots \sum_{k_s=0}^\infty \frac{1}{\prod_{i=1}^s(p_i^{1/d})^{k_i}}\\ &= \prod_{i=1}^s \sum_{k_i=0}^\infty \frac{1}{(p_i^{1/d})^{k_i}}\\ &= \prod_{i=1}^s \frac{1}{1-p_i^{-1/d}} < + \infty. \qedhere
		\end{align*}
	\end{proof}
	Using Proposition~\ref{prop:finite_series} and the fact that
	\[\sum_{p \text{ prime}} \frac{1}{p} = +\infty\] we have the following
	\begin{corollary} \label{corollary:primes_not_in_S2N} For any natural numbers $N,d \geq 1$ there are infinitely many primes not in $\mathcal S_{d,N}$. \hfill \qed
	\end{corollary}
	
	\section{Attaining upper bounds} \label{sec:lower_bounds}
		
		In this section, we collect consequences of Theorem~\ref{theorem:key_example}, which in turn imply Theorem~\ref{theorem:lower_bounds}.
				
	\begin{corollary} \label{corollary:cor_key_example} Let $q,r,s \geq 11$ be pairwise distinct primes. 
		\begin{enumerate}
			\item \label{corollary:cor_key_minimal} If $G' \in \Tree$ is at commability distance at least $3$ from $H^{2,2,2}_{q,r,s}$ then any commation of shortest length is of the form $H^{2,2,2}_{q,r,s} \nearrow \nwarrow \ldots G'$.
			\item \label{corollary:cor_key_minimal_tree} Moreover, there is a commation of the same length that starts by $H^{2,2,2}_{q,r,s} \nearrow \Aut(T^{2,2,2}_{q,r,s}) \nwarrow \ldots G'$.
			\item \label{corollary:cor_key_distance} If $p,q,r$ are not in $\mathcal S_{2,N}$ and $G' \in \Tree$ acts geometrically on a tree with degree bounded above by $N$, then $G'$ is at distance at least $3$ from $H^{2,2,2}_{q,r,s}$.
		\end{enumerate}
	\end{corollary}
	\begin{proof}
		\begin{enumerate}
			\item Clearly directions of arrows must alternate in a commation of shortest length. Suppose a commation of shortest length is of the form $H^{2,2,2}_{q,r,s} \nwarrow \nearrow \nwarrow G \ldots G'$. Then by Theorem~\ref{theorem:key_example}\eqref{theorem:key_example_minimal} there is a commation of the form $H^{2,2,2}_{q,r,s}\nearrow \nwarrow G \ldots G'$, contradicting the minimality of the commation.
			\item Start with a commation of shortest length $H^{2,2,2}_{q,r,s}\nearrow \nwarrow G \ldots G'$.  Adding an isomorphism as the first arrow $H^{2,2,2}_{q,r,s} \nwarrow H^{2,2,2}_{q,r,s}\nearrow \nwarrow G \ldots G'$ allows to apply Theorem~\ref{theorem:key_example}\eqref{theorem:key_example_minimal} to the first three arrows, so that there is a commation $H^{2,2,2}_{q,r,s} \nearrow \Aut(T^{2,2,2}_{q,r,s}) \nwarrow G \ldots G'$.
			\item If $G'$ is at distance at most $2$ from $H^{2,2,2}_{q,r,s}$, then adding one or more isomorphisms either at the beginning or at the end of a shortest commation produces a commation $H^{2,2,2}_{q,r,s} \nwarrow \nearrow \nwarrow G'$. But by Theorem~\ref{theorem:key_example}\eqref{theorem:key_example_rigidity} the group $G'$ cannot act geometrically on a tree of degree bounded above by $N$, contrary to the hypothesis.\qedhere
		\end{enumerate}
	\end{proof}
	\begin{corollary}\label{corollary:lower_bound_diameter} Let $q,r,s \geq 11$ be pairwise distince primes, let $N = \max\{q,r,s\}+2$ and let $q',r',s' (\geq 11)$ be pairwise distinct primes not in $\mathcal S_{2,N}$. Then there is no commation of length less than $6$ from $H^{2,2,2}_{q,r,s}$ to $H^{2,2,2}_{q',r',s'}$.
	\end{corollary}
	\begin{proof} Since $H^{2,2,2}_{p,q,r}$ acts on a tree with degree bounded by $N$, it follows from Corollary~\ref{corollary:cor_key_example}\eqref{corollary:cor_key_distance} that $H^{2,2,2}_{q,r,s}$ and $H^{2,2,2}_{q',r',s'}$ are at distance at least $3$. By Corollary~\ref{corollary:cor_key_example}\eqref{corollary:cor_key_minimal} any shortest commation is of the form $\nearrow \nwarrow \ldots \nearrow \nwarrow$ and in particular is of even length. It therefore remains to rule out the possibility that $H^{2,2,2}_{q,r,s}$ and $H^{2,2,2}_{q',r',s'}$ are at distance $4$. If this is the case, then by Corollary~\ref{corollary:cor_key_example}\eqref{corollary:cor_key_minimal_tree} there is a commation \[H^{2,2,2}_{q,r,s} \nearrow \Aut(T^{2,2,2}_{q,r,s}) \nwarrow G \nearrow \nwarrow H^{2,2,2}_{q',r',s'}.\] But this means that $G$ acts geometrically on the tree $T^{2,2,2}_{q,r,s}$ in contradiction with Corollary~\ref{corollary:cor_key_example}\eqref{corollary:cor_key_distance}.
	\end{proof}

		\begin{corollary} \label{corollary:lower_bound_radius}Let $G \in \Tree$ and let $T$ be a geometric $G$-tree. Let $N$ be the maximal degree of a vertex of $T$, and let $q,r,s$ be primes not in $\mathcal S_{2,N}$. Then $G$ is at distance at least $4$ from $H^{2,2,2}_{q,r,s}$.
		\end{corollary}
		\begin{proof} Observe first that $G$ is at distance at least $3$ from $H^{2,2,2}_{q,r,s}$ by Corollary~\ref{corollary:cor_key_example}\eqref{corollary:cor_key_distance}. Thus we assume by contradiction that $G$ is at distance exactly $3$ from $H^{2,2,2}_{q,r,s}$. By Corollary~\ref{corollary:cor_key_example}\eqref{corollary:cor_key_minimal} there is a commation $H^{2,2,2}_{q,r,s} \nearrow \nwarrow G' \nearrow G$. But since $G$ acts geometrically on a tree of degree bounded above by $N$, so does $G'$. This contradicts Corollary~\ref{corollary:cor_key_example}\eqref{corollary:cor_key_distance}.
		\end{proof}
		
		\begin{corollary}\label{corollary:lower_bound_unimod} Let $U \in \Tree$ be a unimodular group and $N_U \in \NN$ be as in Proposition~\ref{prop:bound_on_degree}. For pairwise distinct primes $q,r,s \geq 11$ not in $\mathcal S_{2,N_U}$ the group $H^{2,2,2}_{q,r,s}$ is at distance at least $5$ from $U$.
	\end{corollary}
		\begin{proof} Recall from Proposition~\ref{prop:bound_on_degree} that there is a bound $N$ on the degree of any vertex of any minimal geometric $U$-tree. 
		By Corollary~\ref{corollary:cor_key_example}\eqref{corollary:cor_key_distance} the groups $U$ and $H^{2,2,2}_{q,r,s}$ are at distance at least $3$. Seeking a contradiction, we suppose that they are at distance at most $4$. Using Corollary~\ref{corollary:cor_key_example}\eqref{corollary:cor_key_minimal}
		we find a commation \[ H^{2,2,2}_{q,r,s} \nearrow \nwarrow G \nearrow G' \nwarrow U.\]
		
		Let $T$ be a minimal geometric $G'$-tree. In particular $T$ is a geometric minimal $U$-tree. Hence every vertex of $T$ has degree bounded by $N$ by definition of $N$. But Corollary~\ref{corollary:cor_key_example}\eqref{corollary:cor_key_distance} asserts that $G$ cannot act geometrically on $T$, which is a contradiction.
		\end{proof}

	\bibliographystyle{amsalpha}
	\bibliography{biblio}

\providecommand{\bysame}{\leavevmode\hbox to3em{\hrulefill}\thinspace}
\providecommand{\MR}{\relax\ifhmode\unskip\space\fi MR }
\providecommand{\MRhref}[2]{%
  \href{http://www.ams.org/mathscinet-getitem?mr=#1}{#2}
}
\providecommand{\href}[2]{#2}
\begin{thebibliography}{CdCMT}

\bibitem[Bas93]{Bass93}
Hyman Bass, \emph{Covering theory for graphs of groups}, J. Pure Appl. Algebra
  \textbf{89} (1993), no.~1-2, 3--47. \MR{1239551 (94j:20028)}

\bibitem[BK90]{BassKulkarni90}
Hyman Bass and Ravi Kulkarni, \emph{Uniform tree lattices}, J. Amer. Math. Soc.
  \textbf{3} (1990), no.~4, 843--902. \MR{1065928 (91k:20034)}

\bibitem[BL94]{BassLubotzky94}
Hyman Bass and Alexander Lubotzky, \emph{Rigidity of group actions on locally
  finite trees}, Proc. London Math. Soc. (3) \textbf{69} (1994), no.~3,
  541--575. \MR{1289863 (95k:20033)}

\bibitem[CD14]{CaretteDreesen12}
Mathieu Carette and Dennis Dreesen, \emph{Locally compact convergence groups
  and $n$-transitive actions}, Mathematische Zeitschrift \textbf{278} (2014),
  no.~3-4, 795--827.

\bibitem[CdCMT]{CCMT12}
P.-E. Caprace, Y.~de~Cornulier, N.~Monod, and R.~Tessera, \emph{Amenable
  hyperbolic groups}, to appear in J. Eur. Math. Soc.

\bibitem[Cor]{Cornulier13}
Yves Cornulier, \emph{Commability and focal locally compact groups}, to appear
  in Indiana Univ. Math. J.

\bibitem[Cor12]{Cornulier12}
\bysame, \emph{On the quasi-isometric classification of focal hyperbolic
  groups}, Available at \url{http://www.normalesup.org/~cornulier/math.html},
  2012.

\bibitem[CT]{CaretteTessera13}
Mathieu Carette and Romain Tessera, \emph{Geometric rigidity and flexibility
  for groups acting on trees}, In preparation.

\bibitem[For02]{Forester02}
Max Forester, \emph{Deformation and rigidity of simplicial group actions on
  trees}, Geom. Topol. \textbf{6} (2002), 219--267 (electronic). \MR{1914569
  (2003m:20030)}

\bibitem[KL09]{KleinerLeeb09}
Bruce Kleiner and Bernhard Leeb, \emph{Induced quasi-actions: a remark}, Proc.
  Amer. Math. Soc. \textbf{137} (2009), no.~5, 1561--1567. \MR{2470813
  (2010e:20065)}

\bibitem[MSW02]{MosherSageevWhyte02}
Lee Mosher, Michah Sageev, and Kevin Whyte, \emph{Maximally symmetric trees},
  Geom. Dedicata \textbf{92} (2002), 195--233, Dedicated to John Stallings on
  the occasion of his 65th birthday. \MR{1934019 (2004f:20054)}

\bibitem[MSW03]{MosherSageevWhyte03}
\bysame, \emph{Quasi-actions on trees. {I}. {B}ounded valence}, Ann. of Math.
  (2) \textbf{158} (2003), no.~1, 115--164. \MR{1998479 (2004h:20055)}

\bibitem[Ser80]{Serre80}
Jean-Pierre Serre, \emph{Trees}, Springer-Verlag, Berlin, 1980, Translated from
  the French by John Stillwell. \MR{607504 (82c:20083)}

\bibitem[Tit70]{Tits70}
Jacques Tits, \emph{Sur le groupe des automorphismes d'un arbre}, Essays on
  topology and related topics ({M}\'emoires d\'edi\'es \`a {G}eorges de
  {R}ham), Springer, New York, 1970, pp.~188--211. \MR{0299534 (45 \#8582)}

\end{thebibliography}
	
\end{document}